\documentclass[reqno]{amsart}
\usepackage{amssymb,eucal,latexsym, mathrsfs,xy,enumerate,graphicx,url}

\xyoption{all}

\newenvironment{enumeratei}{\begin{enumerate}[\upshape (i)]}%
{\end{enumerate}}
{\end{enumerate}}
\newenvironment{enumerater}{\begin{enumerate}[\upshape (1)]}%
{\end{enumerate}}

\usepackage[usenames]{color}

\newcommand{\oc}[1]{\left({#1}\right]}

\newcommand{\pup}[1]{\textup{(}{#1}\textup{)}}
\newcommand{\jirr}{join-ir\-re\-duc\-i\-ble}
\newcommand{\mirr}{meet-ir\-re\-duc\-i\-ble}

\newcommand{\jirry}{join-ir\-re\-duc\-i\-bil\-i\-ty}

\newcommand{\jh}{join-ho\-mo\-mor\-phism}
\newcommand{\mh}{meet-ho\-mo\-mor\-phism}
\newcommand{\eqdef}{\underset{\mathrm{def}}{=}}

\DeclareMathOperator{\Dim}{Dim}
\DeclareMathOperator{\Spec}{Spec}
\DeclareMathOperator{\Specl}{Spec_{\ell}}
\DeclareMathOperator{\Specr}{Spec_{r}}
\DeclareMathOperator{\supp}{supp}

\newcommand{\FL}{\operatorname{F}_{\ell}}

\DeclareMathOperator{\conv}{conv}
\DeclareMathOperator{\cone}{cone}
\DeclareMathOperator{\tcl}{cl}
\DeclareMathOperator{\tin}{int}

\DeclareMathOperator{\Ji}{Ji}
\DeclareMathOperator{\Mi}{Mi}

\newcommand{\ga}{\alpha}
\newcommand{\gb}{\beta}

\newcommand{\gd}{\delta}
\newcommand{\gf}{\varphi}

\newcommand{\gl}{\lambda}

\newcommand{\gs}{\sigma}
\newcommand{\go}{\omega}
\newcommand{\eps}{\varepsilon}

\newcommand{\fin}[1]{[{#1}]^{<\omega}}

\newcommand{\bck}[1]{[\![{#1}]\!]}

\newcommand{\gL}{\Lambda}

\newcommand{\gS}{\Sigma}

\newcommand{\sd}{\mathbin{\smallsetminus}}

\newcommand{\jz}{$(\vee,0)$}
\newcommand{\js}{join-semi\-lat\-tice}

\newcommand{\jzs}{\jz-semi\-lat\-tice}

\newcommand{\two}{\mathbf{2}}
\newcommand{\three}{\mathbf{3}}

\newcommand{\ol}[1]{\overline{#1}}

\newcommand{\pI}[1]{\bigl({#1}\bigr)}
\newcommand{\pII}[1]{\Bigl({#1}\Bigr)}

\newcommand{\set}[1]{\left\{#1\right\}}
\newcommand{\setm}[2]{\set{{#1}\mid{#2}}}
\newcommand{\vecm}[2]{\left({#1}\mid{#2}\right)}
\newcommand{\seq}[1]{\langle{#1}\rangle}

\newcommand{\id}{\mathrm{id}}

\newcommand{\es}{\varnothing}
\newcommand{\res}{\mathbin{\restriction}}

\newcommand{\EE}{\mathbb{E}}

\newcommand{\ZZ}{\mathbb{Z}}
\newcommand{\QQ}{\mathbb{Q}}
\newcommand{\RR}{\mathbb{R}}

\DeclareMathOperator{\Bool}{Bool}
\DeclareMathOperator{\Clos}{Clos}
\DeclareMathOperator{\Op}{Op}
\newcommand{\Ops}{\Op^-}

\DeclareMathOperator{\Idc}{Id_c}

\newcommand{\cC}{{\mathcal{C}}}

\newcommand{\cH}{{\mathcal{H}}}

\newcommand{\cKo}{{\overset{\circ}{\mathcal{K}}}}

\newcommand{\cL}{{\mathcal{L}}}

\newcommand{\cO}{{\mathcal{O}}}
\newcommand{\cP}{{\mathcal{P}}}

\newcommand{\cX}{{\mathcal{X}}}

\numberwithin{equation}{section}

\newtheorem*{stat}{\name}
\newcommand{\name}{testing}

\newenvironment{all}[1]{\renewcommand{\name}{#1}\begin{stat}}
                        {\end{stat}}

\theoremstyle{plain}

\newtheorem{theorem}{Theorem}[section]
\newtheorem{proposition}[theorem]{Proposition}
\newtheorem{corollary}[theorem]{Corollary}
\newtheorem{lemma}[theorem]{Lemma}
\newtheorem{examplepf}[theorem]{Example}
\newtheorem{claim}{Claim}

\theoremstyle{definition}

\newtheorem{definition}[theorem]{Definition}
\newtheorem{notation}[theorem]{Notation}
\newtheorem{example}[theorem]{Example}
\newtheorem{problem}{Problem}

\theoremstyle{remark}
\newtheorem{remark}[theorem]{Remark}
\newtheorem*{note}{Note}

\newcommand{\qedc}{{\qed}~{\rm Claim~{\theclaim}.}}
\newcommand{\qedsc}{{\qed}~{\rm Claim.}}

\newenvironment{cproof}
{\begin{proof}[Proof of Claim.]}
{\qedc\renewcommand{\qed}{}\end{proof}}

\numberwithin{figure}{section}
\numberwithin{table}{section}

\newcommand{\ba}{\boldsymbol{a}}
\newcommand{\bb}{\boldsymbol{b}}

\newcommand{\be}{\boldsymbol{e}}
\newcommand{\bbf}{\boldsymbol{f}}

\newcommand{\bx}{\boldsymbol{x}}
\newcommand{\by}{\boldsymbol{y}}
\newcommand{\bz}{\boldsymbol{z}}

\newcommand{\bB}{\boldsymbol{B}}

\newcommand{\bD}{\boldsymbol{D}}

\newcommand{\va}{\mathsf{a}}
\newcommand{\vb}{\mathsf{b}}

\newcommand{\sJ}{\mathsf{J}}

\newcommand{\scL}{\mathbin{\mathscr{L}}}

\title[Spectral spaces]{Spectral spaces of countable Abelian lattice-ordered groups}

\author[F. Wehrung]{Friedrich Wehrung}
\address{LMNO, CNRS UMR 6139\\
D\'epartement de Math\'ematiques\\
Universit\'e de Caen Normandie\\
14032 Caen cedex\\
France}
\email{friedrich.wehrung01@unicaen.fr}
\urladdr{https://wehrungf.users.lmno.cnrs.fr}

\date{\today}



\begin{document}

\begin{abstract}
It is well known that the \emph{$\ell$-spectrum} of an Abelian $\ell$-group, defined as the set of all its prime $\ell$-ideals with the hull-kernel topology, is a completely normal generalized spectral space.
We establish the following converse of this result.

\begin{all}{Theorem}
Every second countable, completely normal generalized spectral space is homeomorphic to the $\ell$-spectrum of some Abelian $\ell$-group.
\end{all}
We obtain this result by proving that a countable distributive lattice~$D$ with zero is isomorphic to the Stone dual of some $\ell$-spectrum (we say that~$D$ is \emph{$\ell$-rep\-re\-sentable}) if{f} for all $a,b\in D$ there are $x,y\in D$ such that $a\vee b=a\vee y=b\vee x$ and $x\wedge y=0$.
On the other hand, we construct a non-$\ell$-rep\-re\-sentable bounded distributive lattice, of cardinality~$\aleph_1$, with an $\ell$-rep\-re\-sentable countable $\scL_{\infty,\go}$-el\-e\-men\-tary sublattice.
In particular, there is no characterization, of the class of all $\ell$-rep\-re\-sentable distributive lattices, by any class of $\scL_{\infty,\go}$ sentences.
\end{abstract}

\maketitle


\section{Introduction}\label{S:Intro}
A \emph{lattice-ordered group}, or \emph{$\ell$-group} for short, is a group~$G$ endowed with a translation-invariant lattice ordering.
An \emph{$\ell$-ide\-al} of~$G$ is an order-convex, normal $\ell$-subgroup~$I$ of~$G$.
We say that~$I$ is \emph{prime} if $I\neq G$ and $x\wedge y\in I$ implies that either $x\in I$ or $y\in I$, for all $x,y\in G$.
We define the \emph{$\ell$-spec\-trum} of~$G$ as the set~$\Specl{G}$ of all prime $\ell$-ide\-als of~$G$, endowed with the ``hull-kernel'' topology, whose closed sets are exactly the sets $\setm{P\in\Specl{G}}{X\subseteq P}$ for $X\subseteq G$.
Characterizing the topological spaces~$\Specl{G}$, for \emph{Abelian} $\ell$-groups~$G$, is a long-standing open problem, which we shall call the \emph{$\ell$-spectrum problem}.

A topological space~$X$ is \emph{generalized spectral} if it is sober (i.e., every irreducible closed set is the closure of a unique singleton) and the collection of all compact open subsets of~$X$ forms a basis of the topology of~$X$, closed under intersections of any two members.
If, in addition, $X$ is compact, we say that it is \emph{spectral}.
It is well known that the $\ell$-spectrum of any Abelian $\ell$-group is a generalized spectral space;
in addition, this space is \emph{completely normal}, that is, for any points~$x$ and~$y$ in the closure of a singleton~$\set{z}$, either~$x$ is in the closure of~$\set{y}$ or~$y$ is in the closure of~$\set{x}$ (cf. Bigard, Keimel, and Wolfenstein \cite[Ch.~10]{Keim1971}).
Delzell and Madden found in~\cite{DelMad1994} an example of a completely normal spectral space which is not an $\ell$-spectrum.
However, their example is not second countable.
The main aim of the present paper is proving that there is no such counterexample in the second countable case (cf. Theorem~\ref{T:ReprCN}).
We also prove, in Section~\ref{S:NotEltary}, that the class of all Stone dual lattices of $\ell$-spectra is neither closed under products nor under homomorphic images.
We also prove that they have no $\scL_{\infty,\go}$-characterization.

For further background on the $\ell$-spectrum problem and related problems, we refer the reader to Mun\-di\-ci \cite[Problem~2]{Mund2011} (where the $\ell$-spectrum problem is stated in terms of \emph{MV-algebras}), Marra and Mundici~\cite{MaMu2002a,MaMu2002b}, Cignoli and Torrens~\cite{CigTor1996}, Di Nola and Grigolia~\cite{DinGri2004}, Cignoli, Gluschankof, and Lucas~\cite{CGL}, Iberkleid, Mart{\'{\i}}nez, and McGovern~\cite{IMM2011}, Delzell and Madden~\cite{DelMad1994,DelMad1995}, Keimel~\cite{Keim1995}.
Our main reference on $\ell$-groups will be Bigard, Keimel, and Wolfenstein~\cite{BKW}, of which we will mostly follow the notation and terminology.
All our $\ell$-groups will be written additively.
For background on lattice theory, we refer to Gr\"atzer \cite{GGFC,LTF}.
As customary, we denote by~$\rightarrow$, or~$\rightarrow_D$ if~$D$ needs to be specified, the Heyting implication in a Heyting algebra~$D$ (cf. Johnstone~\cite{Johnst1982}): hence $a\rightarrow_Db$ is the largest $x\in D$ such that $a\wedge x\leq b$.

\section{Strategy of the proof}\label{S:Strategy}

\subsection{Reduction to a lattice-theoretical problem; consonance}
\label{Su:RedDLat}
Recall the classical \emph{Stone duality} (cf. Stone~\cite{Stone38a}), between distributive lattices with zero and
$0$-lattice homomorphisms with cofinal%
\footnote{A subset~$X$ in a poset~$P$ is \emph{cofinal} if every element of~$P$ lies below some element of~$X$.}
range on the one hand, generalized spectral spaces and spectral%
\footnote{A map between generalized spectral spaces is \emph{spectral} if the inverse image of any compact open set is compact open.}
maps on the other hand.
This duality sends every distributive lattice~$D$ with zero to the set~$\Spec{D}$ of all its (proper) prime ideals, endowed with the usual hull-kernel topology (cf. Gr\"atzer \cite[\S~2.5]{LTF}, Johnstone \cite[\S~II.3]{Johnst1982}); in the other direction, it sends every generalized spectral space~$X$ to the lattice~$\cKo(X)$ of all its compact open subsets.

Characterizing all $\ell$-spectra of Abelian $\ell$-groups amounts to characterizing all their Stone duals, which are distributive lattices with zero.

Now for every Abelian $\ell$-group~$G$, the Stone dual $\cKo(\Specl{G})$ of~$\Specl{G}$ is isomorphic to the (distributive) lattice~$\Idc{G}$ of all principal $\ell$-ideals of~$G$ (cf. Proposition~1.19, together with Theorem~1.10 and Lemma~1.20, in Keimel \cite{Keim1971}); call such lattices \emph{$\ell$-rep\-re\-sentable}.
Hence, we get

\begin{lemma}\label{L:Sp2Lat}
A topological space~$X$ is homeomorphic to the $\ell$-spec\-trum of an Abelian $\ell$-group if{f} it is generalized spectral and the lattice~$\cKo(X)$ is $\ell$-rep\-re\-sentable.
\end{lemma}

In the sequel, we will denote by~$\seq{x}$, or~$\seq{x}_G$ if~$G$ needs to be specified, the $\ell$-ideal of~$G$ generated by any element~$x$ of an $\ell$-group~$G$.

The lattice-theoretical analogue of complete normality is given as follows.

\begin{definition}\label{D:consonant}
Two elements~$a$ and~$b$, in a distributive lattice~$D$ with zero, are \emph{consonant}, in notation $a\sim_Db$, if there are $x,y\in D$ such that $a\leq b\vee x$, $b\leq a\vee y$, and $x\wedge y=0$.
A subset~$X$ of~$D$ is consonant if every pair of elements in~$X$ is consonant.
We say that~$D$ is \emph{completely normal} if it is a consonant subset of itself.
\end{definition}

The following result is a restatement of Monteiro~\cite[Th\'e\-o\-r\`e\-me~V.3.1]{Mont1954}.

\begin{proposition}\label{P:CNTop2Lat}
A generalized spectral space~$X$ is completely normal if{f} its lattice~$\cKo(X)$, of all compact open subsets, is completely normal.
\end{proposition}

The countable%
\footnote{In the present paper, ``countable'' will always mean ``at most countable''.}
case of the $\ell$-spectrum problem can thus be restated more ambitiously as follows:
\begin{quote}\em
Prove that every countable completely normal distributive lattice with zero is $\ell$-rep\-re\-sentable.
\end{quote}

\subsection{Closed homomorphisms}
\label{Su:closedhoms}

Denote by~$\FL(\go)$ the free Abelian $\ell$-group on the first infinite ordinal $\go\eqdef\set{0,1,2,\dots}$.
Given a countable, bound\-ed, completely normal distributive lattice~$L$, our main goal is to construct a surjective lattice homomorphism $f\colon\Idc{\FL(\go)}\twoheadrightarrow L$ which induces an isomorphism $\Idc\pI{\FL(\go)/I}\cong L$ for a suitable $\ell$-ideal~$I$ of~$\FL(\go)$.
Our next definition introduces the lattice homomorphisms allowing (a bit more than) the latter step, itself contained in Lemma~\ref{L:FactClosed}.

\begin{definition}\label{D:closed}
A \jh~$f\colon A\to B$, between \js{s}~$A$ and~$B$, is \emph{closed} if whenever $a_0,a_1\in A$ and $b\in B$, if $f(a_0)\leq f(a_1)\vee b$, then there exists $x\in A$ such that $a_0\leq a_1\vee x$ and $f(x)\leq b$.
\end{definition}


\begin{lemma}\label{L:FactClosed}
Let~$G$ be an Abelian $\ell$-group, let~$S$ be a distributive lattice with zero, and let $\gf\colon\Idc G\twoheadrightarrow S$ be a closed surjective \jh.
Then $I\eqdef\setm{x\in G}{\gf(\seq{x}_G)=0}$ is an $\ell$-ide\-al of~$G$, and there is a unique isomorphism $\psi\colon\Idc(G/I)\to S$ such that $\psi(\seq{x+I}_{G/I})=\gf(\seq{x}_G)$ for every $x\in G^+$.
\end{lemma}

\begin{proof}
It is straightforward to verify that~$I$ is an $\ell$-ide\-al of~$G$ and that there is a unique map $\psi\colon\Idc(G/I)\to S$ such that $\psi(\seq{x+I}_{G/I})=\gf(\seq{x}_G)$ for every $x\in G^+$.
Since~$\gf$ is a surjective \jh, so is~$\psi$.
It remains to verify that~$\psi$ is an order-embedding.

Let $x,y\in G^+$ such that $\psi(\seq{x+I}_{G/I})\leq\psi(\seq{y+I}_{G/I})$.
This means that $\gf(\seq{x}_G)\leq\gf(\seq{y}_G)$, thus, since~$\gf$ is a closed map, there exists $\bz\in\Idc G$ such that $\seq{x}_G\subseteq\seq{y}_G\vee\bz$ and $\gf(\bz)=0$.
Writing $\bz=\seq{z}_G$, for $z\in G^+$, this means that $z\in I$ and $x\leq ny+nz$ for some positive integer~$n$.
Therefore, $x+I\leq n(y+I)$,  so $\seq{x+I}_{G/I}\subseteq\seq{y+I}_{G/I}$.
\end{proof}

Although this fact will not be used further in the paper, we record here that much of the relevance of closed maps is contained in the following easy result.

\begin{proposition}\label{P:Idcf_closed}
Let~$G$ and~$H$ be Abelian $\ell$-groups and let $f\colon G\to H$ be an $\ell$-homomorphism.
Then the map $\Idc{f}\colon\Idc{G}\to\Idc{H}$, $\seq{x}_G\mapsto\seq{f(x)}_H$ is a closed $0$-lattice homomorphism.
\end{proposition}

\begin{proof}
It is obvious that the map $\bbf\eqdef\Idc{f}$ is a $0$-lattice homomorphism.
Let $\ba_0,\ba_1\in\Idc G$ and let $\bb\in\Idc H$ such that $\bbf(\ba_0)\subseteq\bbf(\ba_1)\vee\bb$.
Pick $a_0,a_1\in G^+$, $b\in H^+$ such that each $\ba_i=\seq{a_i}_G$ and $\bb=\seq{b}_H$.
Then the assumption $\bbf(\ba_0)\subseteq\bbf(\ba_1)\vee\bb$ means that there exists a positive integer~$n$ such that $f(a_0)\leq n(f(a_1)+b)$, which, since $b\geq0$, is equivalent to $(f(a_0)-nf(a_1))^+\leq nb$, that is, since~$f$ is an $\ell$-homomorphism, $f\pI{(a_0-na_1)^+}\leq nb$.
Therefore, setting $\bx\eqdef\seq{(a_0-na_1)^+}_G$, we get $\ba_0\subseteq\ba_1\vee\bx$ and $\bbf(\bx)\subseteq\bb$.
\end{proof}

\begin{example}\label{Ex:Chain2Square}
Using Proposition~\ref{P:Idcf_closed}, it is easy to construct examples of non-$\ell$-rep\-re\-sentable $0,1$-lattice homomorphisms between $\ell$-rep\-re\-sentable finite distributive lattices: for example, consider the unique zero-separating map $\bbf\colon\three\twoheadrightarrow\two$ \pup{where $\two\eqdef\set{0,1}$ and $\three\eqdef\set{0,1,2}$ with their natural orderings}.
\end{example}

\subsection{Elementary blocks: the lattices~$\Ops(\cH)$}
\label{Su:EltaryBl}
Our construction of a closed surjective lattice homomorphism $f\colon\Idc{\FL(\go)}\twoheadrightarrow D$ will be performed stepwise, by expressing~$\Idc{\FL(\go)}$ as a countable ascending union $\bigcup_{n<\go}E_n$, for suitable finite sublattices~$E_n$ (the ``elementary blocks'' of the construction) and homomorphisms $f_n\colon E_n\to L$, then extending each~$f_n$ to~$f_{n+1}$.
Each step of the construction will be one of the following:
\begin{enumerater}
\item\label{ExtDom}
extend the domain of~$f_n$ --- in order to get the final map~$f$ defined on all of~$\Idc{\FL(\go)}$;
this will be done in Section~\ref{S:JirrOpcH}, \emph{via} a lattice-theoretical homomorphism extension result (Lemma~\ref{L:Suff41stepCNpure}) established in Section~\ref{S:HomExt};

\item\label{ExtClos}
correct ``closure defects'' of~$f_n$ (i.e., $f_n(a_0)\leq f_n(a_1)\vee b$ with no~$x$ such that $a_0\leq a_1\vee x$ and $f_n(x)\leq b$) --- in order to get~$f$ closed (Section~\ref{S:ForcCl});

\item\label{ExtSurj}
add elements to the range of~$f_n$ --- in order to get~$f$ surjective (Section~\ref{S:ExtSpecial}).
\end{enumerater}

Elaborating on the final example in Di Nola and Grigolia~\cite{DinGri2004}, it can be seen that not all the~$E_n$ can be taken completely normal.
Our~$E_n$ will be defined as sublattices, of the powerset lattice of an infinite-dimensional vector space~$\RR^{(\go)}$, generated by open half-spaces arising from finite collections of hyperplanes.
Those lattices will be denoted in the form~$\Ops(\cH)$ (cf. Notation~\ref{Not:BoolOpcH} and Lemma~\ref{L:OpscH}).
This will be made possible by the Baker-Beynon duality.

While Steps~\eqref{ExtDom} and~\eqref{ExtClos} above require relatively complex arguments, they remain valid with~$\RR^{(\go)}$ replaced by~$\RR^d$ for any positive integer~$d$, and in fact any topological vector space.
On the other hand, while the argument handling Step~\eqref{ExtSurj} is noticeably easier, it requires an infinite-dimensional ambient space.

\section{Difference operations}
\label{S:DiffOp}

The present section consists of a few technical lattice-theoretical results, mostly aimed at Lemmas~\ref{L:Suff41stepCNpure} and~\ref{L:1stepForcCl}, describing how the concept of a \emph{difference operation} (Definition~\ref{D:DiffFunct}) works in the presence of consonance.

We denote by~$\Ji L$ (resp., $\Mi L$) the set of all \jirr\ (resp., \mirr) elements in a lattice~$L$.
For $p\in L$, we denote by~$p_*$ the largest element of~$L$ smaller than~$p$ --- also called the \emph{lower cover} of~$p$ (cf. Gr\"atzer \cite[p.~102]{LTF}).
If~$L$ is finite, then~$p_*$ exists if{f}~$p\in\Ji{L}$.
We first state a preparatory lemma.

\begin{lemma}[folklore; see Exercises~8.5 and~8.6 in Davey and Priestley~\cite{DavPri1990}]\label{L:pdagger}
Let~$D$ be a finite distributive lattice.
Then every \jirr\ element~$p$ of~$D$ is \emph{join-prime}, that is, it is nonzero and $p\leq x\vee y$ implies that $p\leq x$ or $p\leq y$, for all $x,y\in D$.
Moreover, the subset $\setm{x\in D}{p\nleq x}$ has a largest element~$p^{\dagger}$.
The assignment $p\mapsto p^{\dagger}$ defines an order-isomorphism from~$\Ji D$ onto~$\Mi D$.
\end{lemma}

We now introduce one of our main lattice-theoretical concepts.

\begin{definition}\label{D:DiffFunct}
Let~$L$ be a lattice and let~$S$ be a \jzs.
A map $L\times L\to S$, $(x,y)\mapsto x\sd y$ is an \emph{$S$-valued difference operation} on~$L$ if the following statements hold:
\begin{itemize}
\item[(D0)]
$x\sd x=0$, for all $x\in L$.

\item[(D1)]
$x\sd z=(x\sd y)\vee(y\sd z)$, for all $x,y,z\in L$ such that $x\geq y\geq z$.

\item[(D2)]
$x\sd y=(x\vee y)\sd y=x\sd(x\wedge y)$, for all $x,y\in L$.

\end{itemize}

%
\end{definition}

Although we will need the following lemma only in case~$L$ is distributive, we found it worth noticing that it holds in full generality.

\begin{lemma}\label{L:TrIneq}
Let~$L$ be a lattice, let~$S$ be a \jzs, and let~$\sd$ be an $S$-valued difference operation on~$L$.
Then $x\sd z\leq(x\sd y)\vee(y\sd z)$, for all $x,y,z\in S$ \pup{\emph{triangle inequality}}.
Furthermore, the map $(x,y)\mapsto x\sd y$ is order-preserving in~$x$ and order-reversing in~$y$.
\end{lemma}

\begin{proof}
As in~\cite{WDim}, denote by~$\Delta(x,y)$, for $x\leq y$ in~$L$, the canonical generators of the \emph{dimension monoid}~$\Dim L$ of~$L$.
By the universal property defining~$\Dim L$, our axioms (D0)--(D2) ensure that there is a unique monoid homomorphism $\mu\colon\Dim L\to\nobreak S$ such that $\mu(\Delta(x,y))=y\sd x$ for all $x\leq y$ in~$L$.
Set $\Delta^+(x,y)\eqdef\Delta(x\wedge y,x)$, now for all $x,y\in L$.

Denoting by~$\leq$ the algebraic preordering of~$\Dim L$ (i.e., $\ga\leq\gb$ if there exists~$\gamma$ such that $\gb=\ga+\gamma$), we established in \cite[Prop.~1.9]{WDim} that $\Delta^+(x,z)\leq\Delta^+(x,y)+\Delta^+(y,z)$, for all $x,y,z\in L$.
By taking the image of that inequality under the monoid homomorphism~$\mu$ and using~(D2), the triangle inequality follows.

%
%

Now let $x_1,x_2,y\in L$ with $x_1\leq x_2$.
{}From~(D2) and~(D0) it follows that
 \[
 x_1\sd x_2=x_1\sd(x_1\wedge x_2)=x_1\sd x_1=0\,,
 \]
thus, by the triangle inequality, $x_1\sd y\leq(x_1\sd x_2)\vee(x_2\sd y)=x_2\sd y$.
The proof, that $y_1\leq y_2$ implies $x\sd y_2\leq x\sd y_1$, is similar.
\end{proof}

\begin{lemma}\label{L:D(p_*,p)2D(a,b)}
Let~$L$ be a finite lattice, let~$S$ be a \jzs, and let~$\sd$ be an~$S$-valued difference operation on~$L$.
Then the following statement holds:
 \begin{equation}\label{Eq:D(p_*,p)2D(a,b)}
 a\sd b=\bigvee\vecm{p\sd p_*}
 {p\in\Ji{L}\,,\ p\leq a\,,\ p\nleq b}\,,\quad
 \text{for all }a,b\in L\,.
 \end{equation}
 \end{lemma}

\begin{proof}
Since neither side of~\eqref{Eq:D(p_*,p)2D(a,b)} is affected by changing the pair $(a,b)$ to $(a,a\wedge b)$, we may assume that~$a\geq b$, and then prove~\eqref{Eq:D(p_*,p)2D(a,b)} by induction on~$a$.
The result is trivial for $a=b$ (use~(D0)).
Dealing with the induction step, suppose that $a>b$.
Pick~$a'\in L$ such that~$b\leq a'$ and~$a'$ is a lower cover of~$a$.
The set $\setm{x\in L}{x\leq a\text{ and }x\nleq a'}$ has a minimal element~$p$.
Necessarily, $p$ is \jirr\ and~$p_*\leq a'$, so $a=p\vee a'$ and $p_*=p\wedge a'$.
By~(D2), $p\sd p_*=a\sd a'$.
Moreover, by the induction hypothesis,
 \[
 a'\sd b=\bigvee\vecm{q\sd q_*}
 {q\in\Ji L\,,\ q\leq a'\,,\ q\nleq b}\,.
 \]
Using~(D1), we get
$a\sd b=(a\sd a')\vee(a'\sd b)\geq\bigvee\vecm{q\sd q_*}
 {q\in\Ji L\,,\ q\leq a\,,\ q\nleq b}$.
For the converse inequality, let $q\in\Ji L$ such that~$q\leq a$ and~$q\nleq b$.
Observing that $q\wedge b\leq q_*<q$ and $b<q\vee b\leq a$, we obtain, by using~(D2) together with the second statement of
 Lemma~\ref{L:TrIneq}, $q\sd q_*\leq q\sd(q\wedge b)=(q\vee b)\sd b\leq a\sd b$.
\end{proof}

\begin{lemma}\label{L:a1-b1leqa-b}
Let~$D$ be a distributive lattice, let~$S$ be a \jzs, and let~$\sd$ be an $S$-valued difference operation on~$D$.
Then for all $a_1,b_1,a_2,b_2\in D$, if $a_1\leq b_1\vee a_2$ and $a_1\wedge b_2\leq b_1$ \pup{within~$D$}, then $a_1\sd b_1\leq a_2\sd b_2$ \pup{within~$S$}.
\end{lemma}

\begin{proof}
{}From $a_1\leq b_1\vee a_2$ if follows that $a_1\leq b_1\vee(a_1\wedge a_2)$, thus
 \begin{align*}
 a_1\sd b_1&\leq\pI{b_1\vee(a_1\wedge a_2)}\sd b_1
 &&\text{(use Lemma~\ref{L:TrIneq})}\\
 &=(a_1\wedge a_2)\sd b_1
 &&(\text{use~(D2)})\\
 &\leq(a_1\wedge a_2)\sd(a_1\wedge b_2)
 &&\text{(by our assumptions and Lemma~\ref{L:TrIneq})}\\
 &=(a_1\wedge a_2)\sd b_2
 &&(\text{use~(D2)})\\
 &\leq a_2\sd b_2
 &&\text{(use Lemma~\ref{L:TrIneq})}.\tag*{\qed} 
 \end{align*}
\renewcommand{\qed}{}
\end{proof}

For any elements~$a$ and~$b$ in a distributive lattice~$D$, we shall set
 \begin{equation}\label{Eq:Defominus}
 a\ominus_Db\eqdef\setm{x\in D}{a\leq b\vee x}\,.
 \end{equation}
Moreover, we shall denote by $a\sd_Db$ the least element of $a\ominus_Db$ if it exists, and then call it the \emph{pseudo-difference} of~$a$ and~$b$ relatively to~$D$.
Further, we shall say that~$D$ is a \emph{generalized dual Heyting algebra} if $a\sd_Db$ exists for all $a,b\in D$.

The following lemma will be a crucial source of difference operations throughout the paper.
The proof is straightforward and we leave it to the reader.

\begin{lemma}\label{L:diffS}
For any generalized dual Heyting algebra~$S$, the operation~$\sd_S$ is an $S$-valued difference operation on~$S$.
\end{lemma}

The two following lemmas state that the pseudo-difference operation behaves especially well in the presence of consonance.

\begin{lemma}\label{L:sdSLatHom}
The following statements hold, for every generalized dual Heyting algebra~$S$ and all $a_1,a_2,a,b_1,b_2,b\in S$:
\begin{enumerater}
\item\label{a1jja2sdSb}
$(a_1\vee a_2)\sd_Sb=(a_1\sd_Sb)\vee(a_2\sd_Sb)$.

\item\label{asdSb1mmb2}
$a\sd_S(b_1\wedge b_2)=(a\sd_Sb_1)\vee(a\sd_S b_2)$.

\item\label{a1mma2sdSb}
If $a_1\sim_Sa_2$, then $(a_1\wedge a_2)\sd_Sb=(a_1\sd_Sb)\wedge(a_2\sd_Sb)$.

\item\label{asdSb1jjb2}
If $b_1\sim_Sb_2$, then $a\sd_S(b_1\vee b_2)=(a\sd_Sb_1)\wedge(a\sd_Sb_2)$.
\end{enumerater}
\end{lemma}

\begin{proof}
The proofs of~\eqref{a1jja2sdSb} and \eqref{asdSb1mmb2} are easy exercises.
%

\emph{Ad}~\eqref{a1mma2sdSb}.
We first compute as follows:
 \begin{align*}
 a_1\sd_Sb&\leq\pI{a_1\sd_S(a_1\wedge a_2)}
 \vee\pI{(a_1\wedge a_2)\sd_Sb}
 &&(\text{use Lemmas~\ref{L:TrIneq} and~\ref{L:diffS}})\\
 &=(a_1\sd_Sa_2)
 \vee\pI{(a_1\wedge a_2)\sd_Sb}
 &&(\text{use~(D2)}).
 \end{align*}
Symmetrically, $a_2\sd_Sb\leq(a_2\sd_Sa_1)
 \vee\pI{(a_1\wedge a_2)\sd_Sb}$.
By meeting the two inequalities, we obtain, by using the distributivity of~$S$, the following inequality:
  \[
  (a_1\sd_Sb)\wedge(a_2\sd_Sb)\leq
  \pI{(a_1\sd_Sa_2)\wedge(a_2\sd_Sa_1)}\vee
  \pI{(a_1\wedge a_2)\sd_Sb}\,.
  \]
Now our assumption $a_1\sim_Sa_2$ means that
$(a_1\sd_Sa_2)\wedge(a_2\sd_Sa_1)=0$, so we obtain the following inequality:
  \[
  (a_1\sd_Sb)\wedge(a_2\sd_Sb)\leq
  (a_1\wedge a_2)\sd_Sb\,.
  \]
The converse inequality is trivial.

The proof of~\eqref{asdSb1jjb2} is similar to the one of~\eqref{a1mma2sdSb}.
\end{proof}

\begin{lemma}\label{L:a1nsdb1nmm20}
Let~$S$ be a generalized dual Heyting algebra and let $a_1,a_2,b_1,b_2\in S$.
If $a_1\sim_Sa_2$ and $a_1\wedge a_2\leq b_1\wedge b_2$, then $(a_1\sd_Sb_1)\wedge(a_2\sd_Sb_2)=0$.
\end{lemma}

\begin{proof}
Set $b\eqdef b_1\wedge b_2$.
We compute as follows:
 \begin{align*}
 (a_1\sd_Sb_1)\wedge(a_2\sd_Sb_2)&\leq
 (a_1\sd_Sb)\wedge(a_2\sd_Sb)
 &&(\text{because each }b_i\geq b)\\
 &=(a_1\wedge a_2)\sd_Sb
 &&(\text{use Lemma~\ref{L:sdSLatHom}})\\
 &=0&&(\text{by assumption})\,.\tag*{\qed}
 \end{align*}
\renewcommand{\qed}{}
\end{proof}

\begin{lemma}\label{L:GenConsonant}
Let~$D$ and~$L$ be distributive lattices, let~$E$ and~$S$ be generalized dual Heyting algebras, and let $g\colon E\to L$ be a lattice homomorphism.
We assume that~$D$ is a sublattice of~$E$, $S$ is a sublattice of~$L$, and~$g[D]$ is a consonant subset of~$S$.
Let~$\gS$ be a subset of~$D$, generating~$D$ as a lattice.
If $g(x\sd_Ey)\leq g(x)\sd_Sg(y)$ for all $x,y\in\gS$, then $g(x\sd_Ey)\leq g(x)\sd_Sg(y)$ for all $x,y\in D$.
\end{lemma}

The situation in Lemma~\ref{L:GenConsonant} is partly illustrated in Figure~\ref{Fig:ConstrSq}.

\begin{figure}[htb]
 \[
 \def\labelstyle{\displaystyle}
 \xymatrix{
 D\ar@{_(->}[rrr]
 \ar[d]^{\text{consonant}}_{g\res D} &&& E\ar[d]^g\\
 S\ar@{^(->}[rrr] &&& L
 }
 \]
\caption{Illustrating Lemma~\ref{L:GenConsonant}}
\label{Fig:ConstrSq}
\end{figure}

\begin{proof}
Let~$x\in\gS$.
We claim that the set $D_x\eqdef\setm{y\in D}{g(x\sd_Ey)\leq g(x)\sd_Sg(y)}$ is equal to~$D$.
Indeed, it follows from our assumptions that $\gS\subseteq D_x$.
For all $y_1,y_2\in D_x$,
 \begin{align*}
 g\pI{x\sd_E(y_1\vee y_2)}&
 \leq g(x\sd_Ey_1)\wedge g(x\sd_Ey_2)
 \ (\text{because }g\text{ is order-preserving})\\
 &\leq\pI{g(x)\sd_Sg(y_1)}\wedge\pI{g(x)\sd_Sg(y_2)}
 \qquad(\text{because }y_1,y_2\in D_x)\\
 &=g(x)\sd_S\pI{g(y_1)\vee g(y_2)}\\
 &\qquad\qquad(\text{because }g[D]
 \text{ is consonant in }S
 \text{ and by Lemma~\ref{L:sdSLatHom}})\\
 &=g(x)\sd_Sg(y_1\vee y_2)\quad
 (\text{because }g\text{ is a \jh})\,,
 \end{align*}
that is, $y_1\vee y_2\in D_x$.
The proof that $y_1\wedge y_2\in D_x$ is similar, although  easier since it does not require any consonance assumption.
Hence, $D_x$ is a sublattice of~$D$.
Since it contains~$\gS$, it contains~$D$; whence $D_x=D$.

This holds for all $x\in\gS$, which means that for all $y\in D$, the set
$D'_y\eqdef\setm{x\in D}{g(x\sd_Ey)\leq g(x)\sd_Sg(y)}$ contains~$\gS$.
Moreover, by an argument similar to the one used in the paragraph above, $D'_y$ is a sublattice of~$D$.
Hence, $D'_y=D$.
This holds for all $y\in D$, as required.
\end{proof}

\section{The Main Extension Lemma for distributive lattices}
\label{S:HomExt}

The key idea, of our proof of Theorem~\ref{T:ReprCN}, is the possibility of extending certain lattice homomorphisms $f\colon D\to L$, where~$D$ and~$L$ are distributive $0$-lattices with~$D$ finite and~$L$ completely normal, to finite, or countable, distributive extensions of~$D$.
The present section is mostly devoted to the required technical lattice-theoretical extension result (Lemma~\ref{L:Suff41stepCNpure}).

We first state a preparatory result.
For any lattice~$D$, denote by $D\ast\sJ_2$ the sublattice of~$D^3$ consisting of all triples $(x,y,z)$ such that $z\leq x$ and $z\leq y$.
The following lemma means that, as the notation suggests, $D\ast\sJ_2$ is, in the bounded case, the free distributive product (cf. Gr\"atzer \cite[Thm.~12.5]{GGFC}) of~$D$ with the second entry~$\sJ_2$ of \emph{Jaskowsky's sequence}, represented in Figure~\ref{Fig:J2}.
For $x,y,z\in D$, the triple $(x\vee z,y\vee z,z)\in D\ast\sJ_2$ can then be identified with $(x\wedge\va)\vee(y\wedge\vb)\vee z$.
The proof of Lemma~\ref{L:DotimesJ2} is a straightforward exercise.

\begin{figure}[htb]
\includegraphics{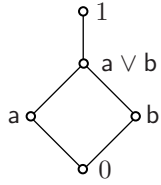}
\caption{The lattice~$\sJ_2$}\label{Fig:J2}
\end{figure}

\begin{lemma}\label{L:DotimesJ2}
Let~$D$ and~$E$ be bounded distributive lattices and let $a,b\in E$ such that $a\wedge b=0$.
Then for every $0,1$-lattice homomorphism $f\colon D\to E$, there exists a unique $0,1$-lattice homomorphism $g\colon D\ast\sJ_2\to E$ such that $g(1,0,0)=a$, $g(0,1,0)=\nobreak b$, and $g(x,x,x)=x$ whenever $x\in D$.
This map is defined by the rule
 \[
 g(x,y,z)\eqdef
 \pI{f(x)\wedge a}\vee\pI{f(y)\wedge b}\vee f(z)\,,
 \qquad\text{for all }(x,y,z)\in D\ast\sJ_2\,.
 \]
\end{lemma}

\begin{lemma}[Main Extension Lemma]\label{L:Suff41stepCNpure}
Let~$E$ be a finite distributive lattice, let~$D$ be a $0,1$-sublattice of~$E$, and let $a,b\in E$ such that the following conditions hold:
\begin{enumerater}
\item\label{EgenDab}
$E$ is generated, as a lattice, by $D\cup\set{a,b}$.

\item\label{DsubHaE}
$D$ is a Heyting subalgebra of~$E$.

\item\label{ammb=0}
$a\wedge b=0$.

\item\label{pleqp*ab}
For all $p\in\Ji{D}$, if $p\leq p_*\vee a\vee b$, then either $p\leq p_*\vee a$ or $p\leq p_*\vee b$.

\item\label{ab2incomppq}
For all~$p,q\in\Ji{D}$, if $p\leq p_*\vee a$ and $q\leq q_*\vee b$, then~$p$ and~$q$ are incomparable.
\end{enumerater}
Let~$L$ be a generalized dual Heyting algebra and let $f\colon D\to L$ be a consonant $0$-lattice homomorphism.
For every $t\in E$, we set
 \begin{align*}
 f_*(t)&\eqdef\bigvee\vecm{f(p)\sd_Lf(p_*)}{p\in\Ji{D}
 \text{ and }p\leq p_*\vee t}\,,\\
 f^*(t)&\eqdef\bigwedge\vecm{f(x)}
 {x\in D\text{ and }t\leq x}\,.
 \end{align*}
Call a pair $(\ga,\gb)$ of elements of~$L$ \emph{$f$-admissible} if there exists a \pup{necessarily unique} lattice homomorphism $g\colon E\to L$, extending~$f$, such that $g(a)=\ga$ and $g(b)=\gb$.
Then the following statements hold:
\begin{enumeratei}
\item
$f_*(t)\leq f^*(t)$ for every $t\in E$.

\item
$f_*(a)\wedge f_*(b)=0$.

\item
The $f$-admissible pairs are exactly the pairs $(\ga,\gb)$ satisfying the inequalities $f_*(a)\leq\ga\leq f^*(a)$, $f_*(b)\leq\gb\leq f^*(b)$, and $\ga\wedge\gb=0$.
\end{enumeratei}
 
\end{lemma}

\begin{note}
Although the proof of our main result (viz. Theorem~\ref{T:ReprCN}) will require only the consideration of $(\ga,\gb)=(f_*(a),f_*(b))$, we keep the more general formulation, due to possible relevance to further extensions of the present work.
The proof of Lemma~\ref{L:Suff41stepCNpure} is mostly  unaffected by that increase in generality.
\end{note}

\begin{proof}
The uniqueness statement on~$g$ follows immediately from Assumption~\eqref{EgenDab}, so we need to deal only with the existence statement.

Since~$f$ is consonant, the assignment $(x,y)\mapsto f(x)\sd_Lf(y)$ defines an $L$-valued difference operation on~$D$ (use Lemma~\ref{L:diffS}).
By Lemma~\ref{L:D(p_*,p)2D(a,b)}, it follows that
 \begin{equation}\label{Eq:Breakf(x)-f(y)}
 f(x)\sd_Lf(y)=\bigvee\vecm{f(p)\sd_Lf(p_*)}{p\in\Ji{D}\,,\ 
 p\leq x\,,\ p\nleq y}\,,\quad\text{for all }x,y\in D\,.
 \end{equation}
The remainder of our proof consists mainly of a series of claims.

\setcounter{claim}{0}

\begin{claim}\label{Cl:newgagb}
For all $x,y\in D$ and all $t\in E$, $x\leq y\vee t$ implies that $f(x)\leq f(y)\vee f_*(t)$.
\end{claim}

\begin{cproof}
Let $x,y\in D$ such that $x\leq y\vee t$, and let $p\in\Ji{D}$ such that $p\leq x$ and $p\nleq y$.
The latter relation means that $y\leq p^{\dagger}$.
Hence, from $x\leq y\vee t$ it follows that $p\leq p^{\dagger}\vee t$, or, equivalently, $p\leq p_*\vee t$.
Therefore, by the definition of~$f_*(t)$, we get $f(p)\sd_Lf(p_*)\leq f_*(t)$.
Joining those inequalities over all possible values of~$p$ and invoking~\eqref{Eq:Breakf(x)-f(y)}, we get $f(x)\sd_Lf(y)\leq f_*(t)$ and we are done.
\end{cproof}

\begin{claim}\label{Cl:xleqyjjajjb}
Let $x,y\in D$ such that $x\leq y\vee a\vee b$.
Then $f(x)\leq f(y)\vee f_*(a)\vee f_*(b)$.
\end{claim}

\begin{cproof}
We must prove that $f(x)\sd_Lf(y)\leq f_*(a)\vee f_*(b)$.
By~\eqref{Eq:Breakf(x)-f(y)}, it suffices to prove that $f(p)\sd_Lf(p_*)\leq f_*(a)\vee f_*(b)$, for every $p\in\Ji{D}$ such that $p\leq x$ and $p\nleq y$.
The latter relation means that $y\leq p^{\dagger}$, thus, for any such~$p$, the inequality $p\leq p^{\dagger}\vee a\vee b$, or, equivalently, $p\leq p_*\vee a\vee b$, holds.
By Assumption~\eqref{pleqp*ab}, this implies that either $p\leq p_*\vee a$ or $p\leq p_*\vee b$.
By the definition of~$f_*(a)$ and~$f_*(b)$, this implies that either $f(p)\sd_Lf(p_*)\leq f_*(a)$ or $f(p)\sd_Lf(p_*)\leq f_*(b)$.
In both cases, $f(p)\sd_Lf(p_*)\leq f_*(a)\vee f_*(b)$.
\end{cproof}

\begin{claim}\label{Cl:ableqv}
$f_*(t)\leq f^*(t)$, for every $t\in E$.
\end{claim}

\begin{cproof}
We must prove that for all~$p\in\Ji{D}$ and all~$x\in D$ such that $p\leq p_*\vee t$ and~$t\leq x$, the inequality $f(p)\sd_Lf(p_*)\leq f(x)$ holds.
Since obviously, $p\leq p_*\vee x$, it follows, since~$p$ is join-prime in~$D$, that $p\leq x$, so we obtain the inequalities $f(p)\sd_Lf(p_*)\leq f(p)\leq f(x)$.
\end{cproof}

\begin{claim}\label{Cl:gammgb=0}
$f_*(a)\wedge f_*(b)=0$.
\end{claim}

\begin{cproof}
It suffices to prove that for all $p,q\in\Ji{D}$ with $p\leq p_*\vee a$ and $q\leq q_*\vee b$, the relation $(f(p)\sd_Lf(p_*))\wedge(f(q)\sd_Lf(q_*))=0$ holds.
By Lemma~\ref{L:a1nsdb1nmm20}, it suffices to prove that $f(p)\wedge f(q)\leq f(p_*)\wedge f(q_*)$.
Since~$f$ is a \mh, it suffices to prove that $p\wedge q\leq p_*\wedge q_*$.
However, it follows from~\eqref{ab2incomppq} that~$p$ and~$q$ are incomparable, so this is obvious.
\end{cproof}

It is clear that every $f$-admissible pair $(\ga,\gb)$ satisfies $f_*(a)\leq\ga\leq f^*(a)$, $f_*(b)\leq\gb\leq f^*(b)$, and $\ga\wedge\gb=0$.
It thus remains to prove that conversely, every such pair $(\ga,\gb)$ is $f$-admissible.

\begin{claim}\label{Cl:extf2joinhom}
There exists a unique map $g\colon E\to L$ such that
 \begin{equation}\label{Eq:g(t)=fxfyfz}
 g\pI{(x\wedge a)\vee(y\wedge b)\vee z}=
 (f(x)\wedge\ga)\vee(f(y)\wedge\gb)\vee f(z)\,,\quad
 \text{for all }x,y,z\in D\,.
 \end{equation}
Moreover, $g(a)=\ga$, $g(b)=\gb$, and~$g$ is a \jh\ extending~$f$.
\end{claim}

\begin{cproof}
By Assumptions~\eqref{EgenDab} and~\eqref{ammb=0}, every element~$t$ of~$E$ has the form $(x\wedge a)\vee(y\wedge b)\vee z$, where $x,y,z\in D$.
This implies the uniqueness statement on~$g$, and says that all we need to do is to verify that the right hand side of~\eqref{Eq:g(t)=fxfyfz} depends only on~$t$;
the map~$g$ thus defined, \emph{via}~\eqref{Eq:g(t)=fxfyfz}, would then automatically be a \jh\ extending~$f$, satisfying, by virtue of the relations $f(0)=0$,  $\ga\leq f(1)$, and $\gb\leq f(1)$, the equations~$g(a)=\ga$ and~$g(b)=\gb$.
Hence, we only need to verify that the following implications hold, for every $u\in D$:
 \begin{align}
 u\leq(x\wedge a)\vee(y\wedge b)\vee z&\Rightarrow
 f(u)\leq(f(x)\wedge\ga)\vee(f(y)\wedge\gb)\vee f(z)\,,
 \label{Eq:uleqxaybz}\\
 u\wedge a\leq
 (x\wedge a)\vee(y\wedge b)\vee z&\Rightarrow
 f(u)\wedge\ga\leq
 (f(x)\wedge\ga)\vee(f(y)\wedge\gb)\vee f(z)\,,
 \label{Eq:ualeqxaybz}\\
 u\wedge b\leq
 (x\wedge a)\vee(y\wedge b)\vee z&\Rightarrow
 f(u)\wedge\gb\leq
 (f(x)\wedge\ga)\vee(f(y)\wedge\gb)\vee f(z)\,,
 \label{Eq:ubleqxaybz}
 \end{align}
for all $u,x,y,z\in D$.
Since~$E$ is distributive, the premise of~\eqref{Eq:uleqxaybz} is equivalent to the conjunction of the following inequalities:
 \begin{align*}
 u&\leq x\vee y\vee z\,;\\
 u&\leq x\vee b\vee z\,;\\
 u&\leq a\vee y\vee z\,;\\
 u&\leq a\vee b\vee z\,. 
 \end{align*}
Since~$f$ is a \jh\ and by Claims~\ref{Cl:newgagb} and~\ref{Cl:xleqyjjajjb}, together with the inequalities $f_*(a)\leq\ga$ and $f_*(b)\leq\gb$, those inequalities imply the following inequalities:
 \begin{align*}
 f(u)&\leq f(x)\vee f(y)\vee f(z)\,;\\
 f(u)&\leq f(x)\vee\gb\vee f(z)\,;\\
 f(u)&\leq\ga\vee f(y)\vee f(z)\,;\\
 f(u)&\leq\ga\vee\gb\vee f(z)\,. 
 \end{align*}
Since~$L$ is distributive, this implies, by reversing the argument above, the inequality $f(u)\leq (f(x)\wedge\ga)\vee(f(y)\wedge\gb)\vee f(z)$, thus completing the proof of~\eqref{Eq:uleqxaybz}.

Further, since~$E$ is distributive and since $a\wedge b=0$, the premise of~\eqref{Eq:ualeqxaybz} is equivalent to 
the inequality $u\wedge a\leq(x\wedge a)\vee z$, thus to the inequality $u\wedge a\leq x\vee z$, which can be written $a\leq(u\rightarrow_E(x\vee z))$.
By Assumption~\eqref{DsubHaE}, this is equivalent to $a\leq v$, where we set $v\eqdef(u\rightarrow_D(x\vee z))$.
Since $\ga\leq f^*(a)$, this implies that $\ga\leq f(v)$.
Hence, $f(u)\wedge\ga\leq f(u)\wedge f(v)=f(u\wedge v)\leq f(x\vee z)=f(x)\vee f(z)$.
Since~$L$ is distributive, this implies in turn that
 \[
 f(u)\wedge\ga\leq(f(x)\wedge\ga)\vee f(z)\leq
 (f(x)\wedge\ga)\vee(f(y)\wedge\gb)\vee f(z)\,
 \]
thus completing the proof of~\eqref{Eq:ualeqxaybz}.
The proof of~\eqref{Eq:ubleqxaybz} is symmetric.
\end{cproof}

In order to conclude the proof of Lemma~\ref{L:Suff41stepCNpure}, it is sufficient to prove that~$g$ is a \mh.
By Assumption~\eqref{ammb=0} and since $\ga\wedge\gb=0$, respectively, it follows from Lemma~\ref{L:DotimesJ2} that there are unique lattice homomorphisms $d\colon D\ast\sJ_2\to E$ and $\gd\colon D\ast\sJ_2\to L$ such that
 \[
 d(x,y,z)=(x\wedge a)\vee(y\wedge b)\vee z\text{ and }
 \gd(x,y,z)=\pI{f(x)\wedge a}\vee\pI{f(y)\wedge b}\vee f(z)
 \]
for all $(x,y,z)\in D\ast\sJ_2$.
Then Claim~\ref{Cl:extf2joinhom} implies that $\gd=g\circ d$.
Moreover, it follows from Assumptions~\eqref{EgenDab} and~\eqref{ammb=0} that~$d$ is surjective.
Now any two elements of~$E$ have the form~$d(t_1)$ and~$d(t_2)$, where $t_1,t_2\in D\ast\sJ_2$, and
 \[
 g(d(t_1))\wedge g(d(t_2))=\gd(t_1)\wedge\gd(t_2)
 =\gd(t_1\wedge t_2)=g(d(t_1\wedge t_2))\leq
  g(d(t_1)\wedge d(t_2))\,.
 \]
The converse inequality
$g(d(t_1)\wedge d(t_2))\leq g(d(t_1))\wedge g(d(t_2))$ is trivial.
\end{proof}

\section{Lattices of convex open polyhedral cones}
\label{S:CxOp}

Throughout this section we shall fix a real topological vector space~$\EE$.
Denote by~$\tin(A)$ and~$\tcl(A)$ the interior and closure of a subset~$A$, respectively.
We begin with two preparatory lemmas.

\begin{lemma}\label{L:cl(AiiF)}
Let~$A$ and~$F$ be convex subsets in~$\EE$, with~$F$ closed and $F\cap\tin(A)\neq\es$.
Then $\tcl(F\cap A)=F\cap\tcl(A)$.
\end{lemma}

\begin{proof}
Fix $u\in F\cap\tin(A)$, and let $p\in F\cap\tcl(A)$.
Since~$F$ and~$A$ are both convex, $(1-\gl)p+\gl u\in F\cap A$ for each $\gl\in\oc{0,1}$.
Since $(1-\gl)p+\gl u$ converges to~$p$, as~$\gl$ goes to~$0$ and~$\gl>0$, it follows that $p\in\tcl(F\cap A)$.
We have thus proved that $F\cap\tcl(A)\subseteq\tcl(F\cap A)$.
The converse containment is trivial.
\end{proof}

\begin{lemma}\label{L:QmmHnwd}
Let~$F$ be the union of finitely many closed subspaces in~$\EE$ and let~$Q$ be a convex subset of~$\EE$.
Then either $Q\subseteq F$ or~$Q\cap F$ is nowhere dense in~$Q$.
\end{lemma}

\begin{proof}
We first deal with the case where~$F$ is a closed subspace of~$\EE$.
Suppose that~$Q\cap F$ is not nowhere dense in~$Q$.
Since~$F$ is a closed subspace of~$\EE$, $Q\cap F$ is relatively closed in~$Q$, thus the relative interior~$U$ of~$Q\cap F$ in~$Q$ is nonempty.
Fix $u\in U$ and let $q\in Q$.
Since~$Q$ is convex, $(1-\gl)u+\gl q\in Q$ for every $\gl\in[0,1]$.
Since~$U$ is a relative neighborhood of~$u$ in~$Q$, it follows that $(1-\gl)u+\gl q$ belongs to~$U$, thus to~$F$, for some $\gl\in\oc{0,1}$.
Since $\set{u,(1-\gl)u+\gl q}\subseteq F$ with $\gl>0$, it follows that $q\in F$, therefore completing the proof that $Q\subseteq F$.

In the general case, $F=\bigcup_{i=1}^nF_i$, where each~$F_i$ is a closed subspace of~$\EE$.
If $Q\not\subseteq F$, then $Q\not\subseteq F_i$, thus, by the paragraph above, $Q\cap F_i$ is nowhere dense in~$Q$ whenever $1\leq i\leq n$; whence $Q\cap F=\bigcup_{i=1}^n(Q\cap F_i)$ is also nowhere dense in~$Q$.
\end{proof}

{}From now on, for any closed hyperplane~$H$ of~$\EE$, we shall denote by~$H^+$ and~$H^-$ the open half-spaces with boundary~$H$, with associated closed half-spaces $\ol{H}^+\eqdef\tcl(H^+)$ and $\ol{H}^-\eqdef\tcl(H^-)$.

\begin{notation}\label{Not:BoolOpcH}
For a set~$\cH$ of closed hyperplanes of~$\EE$, we will set
 \begin{align*}
 \gS_{\cH}&\eqdef\setm{H^+}{H\in\cH}\cup
 \setm{H^-}{H\in\cH}\,,\\
 \ol{\gS}_{\cH}&\eqdef\setm{\ol{H}^+}{H\in\cH}\cup
 \setm{\ol{H}^-}{H\in\cH}\,.
 \end{align*}
Furthermore, we will denote by~$\Bool(\cH)$ the Boolean algebra of subsets of~$\EE$ generated by $\gS_{\cH}$ (equivalently, by~$\ol{\gS}_{\cH}$), and by~$\Clos(\cH)$ (resp., $\Op(\cH)$) the lattice of all closed (resp., open) members of~$\Bool(\cH)$.
\end{notation}

Trivially, $\Clos(\cH)$ and~$\Op(\cH)$ are both $0,1$-sublattices of~$\Bool(\cH)$, which is a $0,1$-sublattice of the powerset lattice of~$\EE$.
For the remainder of this section we shall fix a nonempty%
\footnote{Lemma~\ref{L:CharOpPH} does not extend to $\cH=\es$, for $\Op(\es)=\set{\es,\EE}$ while $\gS_{\es}=\es$.}
set~$\cH$ of closed hyperplanes of~$\EE$ through the origin.

\begin{lemma}\label{L:CharOpPH}
For every $X\in\Bool(\cH)$, the subsets~$\tcl(X)$ and~$\tin(X)$ both belong to $\Bool(\cH)$.
Moreover, $\Op(\cH)$ is generated, as a lattice, by
$\gS_{\cH}\cup\set{\EE}$, and it is Heyting subalgebra of the Heyting algebra~$\cO(\EE)$ of all open subsets of~$\EE$.
\end{lemma}

\begin{proof}
For the duration of the proof, we shall denote by~$\Clos'(\cH)$ (resp., $\Op'(\cH)$) the sublattice of~$\Bool(\cH)$ generated by~$\ol{\gS}_{\cH}\cup\set{\es}$ (resp., $\gS_{\cH}\cup\set{\EE}$).

We first prove that the closure of any member of~$\Bool(\cH)$ belongs to~$\Clos'(\cH)$.
Writing the elements of~$\Bool(\cH)$ in disjunctive normal form, we see that every element of~$\Bool(\cH)$ is a finite union of finite intersections of open half-spaces and closed half-spaces with boundaries in~$\cH$.
Since $H^{\gs}=\ol{H}^{\gs}\setminus H$ for all $H\in\cH$ and all $\gs\in\set{+,-}$, it follows that every element of~$\Bool(\cH)$ is a finite union of sets of the form $Q\setminus F$, where~$Q$ is a finite intersection of closed half-spaces with boundaries in~$\cH$ and~$F$ is a finite union of members of~$\cH$.
Since the closure operator commutes with finite unions, the first statement of Lemma~\ref{L:CharOpPH} thus reduces to verifying that $\tcl(Q\setminus F)$ belongs to~$\Clos'(\cH)$, for any~$Q$ and~$F$ as above.
Now this follows from Lemma~\ref{L:QmmHnwd}:
if $Q\subseteq F$ then $\tcl(Q\setminus F)=\es$, and if $Q\not\subseteq F$, then $Q\cap F$ is nowhere dense in~$Q$, thus $\tcl(Q\setminus F)=\tcl(Q)=Q$.
The statement about the closure follows; in particular, $\Clos'(\cH)=\Clos(\cH)$.
By taking complements, the statement about the interior follows;
in particular, $\Op'(\cH)=\Op(\cH)$.

For all $X,Y\in\Op(\cH)$, the Heyting residue $X\rightarrow Y$, evaluated within the lattice~$\cO(\EE)$ of all open subsets of~$\EE$, is equal to $\tin\pI{(\complement X)\cup Y}$ (where~$\complement$ denotes the complement in~$\EE$), thus, as $(\complement X)\cup Y$ belongs to $\Bool(\cH)$ and by the paragraph above, it belongs to $\Op(\cH)$.
\end{proof}

In particular, the members of~$\Op(\cH)$ are \emph{open polyhedral cones}, that is, finite unions of finite intersections of open half-spaces of~$\EE$.
Lemma~\ref{L:CharOpPH} also says that the topology on~$\EE$ could be, in principle, omitted from the study of~$\Bool(\cH)$ and~$\Op(\cH)$.

%
Define a \emph{basic open member of~$\Op(\cH)$} as a nonempty finite intersection of open half-spaces with boundaries in~$\cH$.
In particular, the intersection of the empty collection yields the basic open set~$\EE$.
Since every element of~$\Op(\cH)$ is a finite union of basic open sets, we obtain the following.

\begin{corollary}\label{C:JirrOpPH}
Every \jirr\ element of~$\Op(\cH)$ is basic open.
In particular, it is convex.
\end{corollary}

It is easy to find examples showing that the converse of Corollary~\ref{C:JirrOpPH} does not hold:
\emph{a basic open member of~$\Op(\cH)$ may not be \jirr}.

\begin{corollary}\label{C:ChangeH}
Let~$H$ be a closed hyperplane of~$\EE$, with associated open half spaces~$H^+$ and~$H^-$.
Then the members of~$\Op(\cH\cup\set{H})$ are exactly the sets of the form $(X\cap H^+)\cup(Y\cap H^-)\cup Z$, where~$X,Y,Z\in\Op(\cH)$;
moreover, one can take~$Z\subseteq X$ and~$Z\subseteq Y$.
\end{corollary}

\begin{proof}
For every basic open set~$U$ in~$\Op(\cH\cup\set{H})$, there is a basic open set~$T$ in~$\Op(\cH)$ such that $U=T\cap H^+$ or $U=T\cap H^-$ or $U=T$.
By Lemma~\ref{L:CharOpPH}, every element of~$\Op(\cH\cup\set{H})$ is a finite union of basic open sets, thus it has the given form.
Moreover, changing~$X$ to~$X\cup Z$ and~$Y$ to~$Y\cup Z$ does not affect the value of $(X\cap H^+)\cup(Y\cap H^-)\cup Z$.\end{proof}

\begin{lemma}\label{L:OpscH}
The top element of~$\Op(\cH)$ \pup{viz.~$\EE$} is \jirr\ in $\Op(\cH)$.
Consequently, the subset $\Ops(\cH)\eqdef\Op(\cH)\setminus\set{\EE}$ is a $0$-sublattice of~$\Op(\cH)$.
It is generated, as a lattice, by~$\gS_{\cH}$.
\end{lemma}

\begin{proof}
Any basic open member of~$\Op(\cH)$, distinct from~$\EE$, omits the origin.
Hence, any member of~$\Op(\cH)$, distinct from~$\EE$, omits the origin, and so the union of any two such sets is distinct from~$\EE$.
This proves that~$\EE$ is \jirr\ in~$\Op(\cH)$.
The verifications of the other statements of Lemma~\ref{L:OpscH} are straightforward.
\end{proof}

\begin{remark}\label{Rk:OpscH}
Let~$\cH$ be finite.
Then the unit of~$\Ops(\cH)$ is equal to~$\complement\bigcap\cH$, which is distinct from the unit of~$\Op(\cH)$, which is equal to~$\EE$.
In particular, $\Ops(\cH)$ is not a Heyting subalgebra of~$\Op(\cH)$.
\end{remark}

\section{The Main Extension Lemma for lattices $\Op(\cH)$}
\label{S:JirrOpcH}

Throughout this section we shall fix a real topological vector space~$\EE$.
Our main goal is to show that Lemma~\ref{L:Suff41stepCNpure} can be applied to lattices of the form~$\Op(\cH)$ (cf. Lemma~\ref{L:OpcHCNpureinOpH'}).
This goal will be achieved \emph{via} a convenient description of the \jirr\ members of~$\Op(\cH)$ (cf. Lemma~\ref{L:JirrOpcH}), involving an operator that we will denote by~$\nabla_{\cH}$ (cf. Notation~\ref{Not:nablaU}).

For any subset~$X$ in~$\EE$, we denote by $\conv(X)$ the convex hull of~$X$, and by~$\cone(X)\eqdef\RR^+\cdot\conv(X)$ the closed convex cone generated by~$X$.
For a set~$X$, a poset~$P$, and maps $f,g\colon X\to P$, we shall set $\bck{f\leq g}\eqdef\setm{x\in X}{f(x)\leq g(x)}$, and similarly for $\bck{f<g}$, $\bck{f>g}$, and so on.
We first state two preparatory lemmas.

\begin{lemma}\label{L:Farkas}
Let~$n$ be a nonnegative integer and let~$b_1$, \dots, $b_n$, $c$ be linear functionals on~$\EE$.
Then $\bigcap_{i=1}^n\bck{b_i\geq0}\subseteq\bck{c\geq0}$ if{f} $c\in\cone(\set{b_1,\dots,b_n})$.
\end{lemma}

\begin{proof}
By working in the quotient space $\EE/{\bigcap_{i=1}^n\ker(b_i)}$,  the problem is reduced to the classical finite-dimensional case (cf. Schrijver \cite[Thm.~7.1]{Sch1986}).
\end{proof}

\begin{lemma}\label{L:FMW}
Suppose that~$\EE$ is Hausdorff.
Then $\cone(X)$ is a closed subset of~$\EE$, for every finite subset~$X$ of~$\EE$.
\end{lemma}

\begin{proof}
The subspace~$F$ of~$\EE$ generated by~$X$ is finite-dimensional, thus (since~$\EE$ is Hausdorff) closed.
This reduces the problem to the case where $\EE=\RR^d$ for some nonnegative integer~$d$.
By the Farkas-Minkowski-Weyl Theorem (cf. Schrijver \cite[Cor.~7.1a]{Schrij1986}), $\cone(X)$ is then a finite intersection of closed half-spaces of~$\EE$.
\end{proof}

Until the end of this section, we will fix a nonempty \emph{finite} set~$\cH$ of closed hyperplanes of~$\EE$ through the origin.

\begin{notation}\label{Not:nablaU}
For every $U\in\Op(\cH)$, we set $\cH_U\eqdef\setm{H\in\cH}{H\cap U\neq\es}$.
The intersection~$\nabla_{\cH}{U}$ of all members of~$\cH_U$ is a closed subspace of~$\EE$.
\end{notation}

Recall (cf. Lemma~\ref{L:pdagger}) that for a \jirr\ member~$P$ of~$\Op(\cH)$, $P^{\dagger}$ denotes the largest element of~$\Op(\cH)$ not containing~$P$.

\begin{lemma}\label{L:JirrOpcH}
A nonempty, convex member~$P$ of~$\Op(\cH)$ is \jirr, within the lattice~$\Op(\cH)$, if{f} $P\cap\nabla_{\cH}{P}$ is nonempty.
Moreover, in that case, the lower cover~$P_*$ of~$P$, in~$\Op(\cH)$, is equal to~$P\setminus\nabla_{\cH}{P}$, and~$P^{\dagger}=\complement(\tcl(P)\cap\nabla_{\cH}{P})$.
\end{lemma}

\begin{proof}
Suppose first that~$P$ is \jirr.
Suppose, by way of contradiction, that $P\cap\nabla_{\cH}{P}=\es$, that is, $P\subseteq\bigcup_{H\in\cH_P}\complement H$.
Since~$P$ is \jirr\ in the distributive lattice~$\Op(\cH)$, it is join-prime in that lattice (cf. Lemma~\ref{L:pdagger}), thus there exists $H\in\cH_P$ such that $P\subseteq\complement H$; in contradiction with $H\in\cH_P$.

Suppose, conversely, that $P\cap\nabla_{\cH}{P}\neq\es$.
The subset $P\setminus\nabla_{\cH}{P}$ belongs to~$\Op(\cH)$ and it is a proper subset of~$P$, thus we only need to prove that every proper subset~$X$ of~$P$, belonging to~$\Op(\cH)$, is contained in~$P\setminus\nabla_{\cH}{P}$.
It suffices to consider the case where~$X$ is basic open.
There are a subset~$\cX$ of~$\cH$ and a family $\vecm{\eps_H}{H\in\cX}$ of elements of~$\set{+,-}$ such that $X=\bigcap_{H\in\cX}H^{\eps_H}$.
Since $P\not\subseteq X$, there exists $H\in\cX$ such that $P\not\subseteq H^{\eps_H}$.
Hence,
 \begin{equation}\label{Eq:PcapepsHaH<=0}
 P\cap\ol{H}^{-\eps_H}\neq\es\,.
 \end{equation}
If $P\subseteq H^{-\eps_H}$, then $X\subseteq H^{-\eps_H}$, thus, since $X\subseteq H^{\eps_H}$, we get $X=\es$, a contradiction.
Hence, $P\not\subseteq H^{-\eps_H}$, that is,
 \begin{equation}\label{Eq:PcapepsHaH>=0}
 P\cap \ol{H}^{\eps_H}\neq\es\,.
 \end{equation}
By~\eqref{Eq:PcapepsHaH<=0} and~\eqref{Eq:PcapepsHaH>=0}, and since~$P$ is convex, it follows that $P\cap H\neq\es$, that is, $H\in\cH_P$.
Hence, $\nabla_{\cH}{P}\subseteq H$.
Since $X\cap H=\es$, it follows that $X\cap\nabla_{\cH}{P}=\es$, that is, $X\subseteq P\setminus\nabla_{\cH}P$, thus completing the proof of the \jirry\ of~$P$.

Finally, it follows from Lemma~\ref{L:CharOpPH} that the set $U\eqdef\tin\complement(P\cap\nabla_{\cH}{P})$ belongs to~$\Op(\cH)$.
Moreover, $U=\complement\tcl(P\cap\nabla_{\cH}{P})$.
Since~$P\cap\nabla_{\cH}{P}\neq\es$ and by Lemma~\ref{L:cl(AiiF)}, we get $U=\complement(\tcl(P)\cap\nabla_{\cH}{P})$.
For every $V\in\Op(\cH)$, $P\not\subseteq V$ if{f} $P\cap V\subsetneqq P$, if{f} $P\cap V\subseteq P_*$, if{f} $P\cap V\cap\nabla_{\cH}{P}=\es$, if{f} $V\subseteq\complement(P\cap\nabla_{\cH}{P})$.
Since~$V$ is open, this is equivalent to $V\subseteq U$.
Therefore, $U=P^{\dagger}$.
\end{proof}


\begin{proposition}\label{P:APantitone}
Let~$P$ and~$Q$ be \jirr\ elements in~$\Op(\cH)$.
If $P\subsetneqq Q$, then $\nabla_{\cH}{Q}\subsetneqq\nabla_{\cH}{P}$.
\end{proposition}

\begin{proof}
By definition, $\cH_P\subseteq\cH_Q$, thus $\nabla_{\cH}{Q}\subseteq\nabla_{\cH}{P}$.
Since $P\subsetneqq Q$ and by Lemma~\ref{L:JirrOpcH}, $P$ is contained in $Q_*=Q\setminus\nabla_{\cH}{Q}$, thus $P\cap\nabla_{\cH}{Q}=\es$.
Since $P\cap\nabla_{\cH}{P}\neq\es$, it follows that $\nabla_{\cH}{P}\neq\nabla_{\cH}{Q}$.
\end{proof}

\begin{lemma}[Extension Lemma for~$\Op(\cH)$]
\label{L:OpcHCNpureinOpH'}
Let~$H$ be a closed hyperplane of~$\EE$, let~$L$ be a generalized dual Heyting algebra, and let $f\colon\Op(\cH)\to L$ be a consonant $0$-lattice homomorphism.
Then~$f$ extends to a unique lattice homomorphism $g\colon\Op(\cH\cup\set{H})\to L$ such that $g(H^+)=f_*(H^+)$ and $g(H^-)=f_*(H^-)$.
\end{lemma}

We refer to Lemma~\ref{L:Suff41stepCNpure} for the notations~$f_*(H^+)$ and~$f_*(H^-)$.

\begin{proof}
It suffices to verify that Conditions \eqref{EgenDab}--\eqref{ab2incomppq} of Lemma~\ref{L:Suff41stepCNpure} are satisfied, with $D:=\Op(\cH)$, $E:=\Op(\cH\cup\set{H})$, $a:= H^+$, and $b:= H^-$.
Conditions~\eqref{EgenDab} (use Corollary~\ref{C:ChangeH}) and~\eqref{ammb=0} are obvious.
By Lemma~\ref{L:CharOpPH}, $\Op(\cH)$ is a Heyting subalgebra of $\Op(\cH\cup\set{H})$;
Condition~\eqref{DsubHaE} follows.

Let~$P$ be a \jirr\ element of~$\Op(\cH)$ such that $P\subseteq P_*\cup H^+\cup H^-$.
By Lemma~\ref{L:JirrOpcH}, this means that $P\cap\nabla_{\cH}{P}\subseteq H^+\cup H^-$.
Since $P\cap\nabla_{\cH}{P}$ is convex, this implies that~$P\cap\nabla_{\cH}{P}$ is contained either in~$H^+$ or in~$H^-$, thus that~$P$ is contained either in $P_*\cup H^+$ or in~$P_*\cup H^-$.
Condition~\eqref{pleqp*ab} follows.

For Condition~\eqref{ab2incomppq}, let $P,Q\in\Ji\Op(\cH)$ such that $P\subseteq P_*\cup H^+$ and $Q\subseteq Q_*\cup H^-$.
Suppose for example that $P\subseteq Q$.
Then $P\cap\nabla_{\cH}{P}\subseteq H^+$, $Q\cap\nabla_{\cH}{Q}\subseteq H^-$, and $P^{\dagger}\subseteq\nobreak Q^{\dagger}$.
Thus, by Lemma~\ref{L:JirrOpcH}, $\tcl(Q)\cap\nabla_{\cH}{Q}\subseteq\tcl(P)\cap\nabla_{\cH}{P}$.
It follows that $Q\cap\nabla_{\cH}{Q}\subseteq\ol{H}^+$, hence $Q\cap\nabla_{\cH}{Q}=\es$, a contradiction.
\end{proof}

\section{Correcting a closure defect}\label{S:ForcCl}

Throughout this section we shall fix a real topological vector space~$\EE$, with topological dual~$\EE'$, endowed with the weak-$*$ topology.

\begin{lemma}\label{L:1stepForcCl}
Let~$\cH$ be a finite set of closed hyperplanes in~$\EE$, let $a,b\in\EE'$ with respective kernels~$A$ and~$B$, both belonging to~$\cH$.
We set
 \begin{align*}
 A^+\eqdef\bck{a>0}\,,&\quad A^-\eqdef\bck{a<0}\,,\\
 B^+\eqdef\bck{b>0}\,,&\quad B^-\eqdef\bck{b<0}\,,\\
 C_m\eqdef\ker(a-mb)\,,&\quad
 \cH_m\eqdef\cH\cup\set{C_m}\,,\\
 C_m^+\eqdef\bck{a>mb}\,,&\quad
 C_m^-\eqdef\bck{a<mb}\,,
 \end{align*}
for any positive integer~$m$.
Then for all large enough~$m$, the following statement holds:
for every generalized dual Heyting algebra~$L$, every consonant $0$-lattice homomorphism $f\colon\Op(\cH)\to L$ extends to a lattice homomorphism $g\colon\Op(\cH_m)\to L$ such that $g(A^+\sd_{\Ops(\cH_m)}B^+)=f(A^+)\sd_Lf(B^+)$.
\end{lemma}

\begin{note}
The notation $A^+\sd_{\Ops(\cH_m)}B^+$ might look a bit crowded, in particular due to the use of~$\Ops(\cH_m)$ instead of~$\Op(\cH_m)$.
In reality, that distinction is immaterial here, because~$\Ops(\cH_m)$ is an ideal of~$\Op(\cH_m)$, thus $U\sd_{\Ops(\cH_m)}V=U\sd_{\Op(\cH_m)}V$ for all $U,V\in\Ops(\cH_m)$.
\end{note}

\begin{proof}
We begin by stating exactly how large~$m$ should be.

\setcounter{claim}{0}

\begin{claim}\label{Cl:HowLargem}
There exists a positive integer~$m_0$ such that for all~$m\geq m_0$ and all $X\in\Op(\cH)$, $C_m^-\subseteq X$ implies that $B^+\subseteq X$.
\end{claim}

\begin{cproof}
Every $P\in\Ji\Op(\cH)$ is basic open, thus both~$\tcl(P)$ and~$\nabla_{\cH}{P}$ are intersections of closed half-spaces with boundaries in~$\cH$.
Hence, there is a finite subset~$\Phi_P$ of $\EE'\setminus\set{0}$ such that $\tcl(P)\cap\nabla_{\cH}{P}=\bigcap_{x\in\Phi_P}\bck{x\geq0}$ and $\ker(x)\in\cH$ for every $x\in\Phi_P$.
Since~$\EE'$ is Hausdorff, it follows from Lemma~\ref{L:FMW} that the closed convex cone~$K_P$ generated by~$\Phi_P$ is a closed subset of~$\EE'$.
Hence, setting $\cP\eqdef\setm{P\in\Ji\Op(\cH)}{-b\notin K_P}$, there exists a positive integer~$m_0$ such that
 \begin{equation}\label{Eq:Defnm0}
 -b+(1/m)a\notin K_P\,,\text{ for all }P\in\cP
 \text{ and all }m\geq m_0\,.
 \end{equation}
It follows from Lemma~\ref{L:Farkas} that for every $y\in\EE$ and every $P\in\Ji\Op(\cH)$, $-y\in\nobreak K_P$ if{f} $\bigcap_{x\in\Phi_P}{\bck{x\geq0}}\subseteq\bck{y\leq0}$, if{f} $\tcl(P)\cap\nabla_{\cH}{P}\subseteq\bck{y\leq0}$, if{f} $\bck{y>0}\subseteq P^{\dagger}$ (cf. Lemma~\ref{L:JirrOpcH}).
In particular, $-b\in K_P$ if{f} $B^+\subseteq P^{\dagger}$.
Similarly, $-b+(1/m)a\in K_P$ if{f} $C_m^-\subseteq P^{\dagger}$.
Hence, \eqref{Eq:Defnm0} means that $C_m^-\subseteq P^{\dagger}$ implies that $B^+\subseteq P^{\dagger}$, whenever $m\geq m_0$ and $P\in\Ji\Op(\cH)$.
Now every \mirr\ element of~$\Op(\cH)$ has the form~$P^{\dagger}$ (cf. Lemma~\ref{L:pdagger}), and every element of~$\Op(\cH)$ is an intersection of \mirr\ elements of~$\Op(\cH)$. 
\end{cproof}

We shall prove that every integer $m\geq m_0$ has the property stated in Lemma~\ref{L:1stepForcCl}.
Let~$L$ be a generalized dual Heyting algebra and let $f\colon\Op(\cH)\to L$ be a consonant $0$-lattice homomorphism.
We consider the extension~$g$ of~$f$, to a homomorphism from~$\Op(\cH_m)$ to~$L$, given by Lemma~\ref{L:OpcHCNpureinOpH'}, with $H:=C_m$, $H^+:=C_m^+$, $H^-:=C_m^-$.
In particular,
 \begin{equation}\label{Eq:g(Cm+)}
 g(C_m^+)=\bigvee\vecm{f(P)\sd_Lf(P_*)}
 {P\in\Ji\Op(\cH)\,,\ P\subseteq P_*\cup C_m^+}\,.
 \end{equation}
We claim that the following inequality holds:
 \begin{equation}\label{Eq:fA+Cm+A+B+}
 f(A^+)\wedge g(C_m^+)\leq f(A^+)\sd_Lf(B^+)\,.
 \end{equation}
Since~$L$ is distributive, this amounts to proving the following statement:
 \begin{multline}\label{Eq:fACm}
 f(A^+)\wedge\pI{f(P)\sd_Lf(P_*)}\leq f(A^+)\sd_Lf(B^+)\,,\\
 \quad\text{for every }P\in\Ji\Op(\cH)\text{ such that }
 P\subseteq P_*\cup C_m^+\,.
 \end{multline}
Let $P\in\Ji\Op(\cH)$ such that $P\subseteq P_*\cup C_m^+$; that is, $P\cap\nabla_{\cH}{P}\subseteq C_m^+$.
It follows that $\tcl(P)\cap\nabla_{\cH}{P}=\tcl(P\cap\nabla_{\cH}{P})\subseteq\ol{C}_m^+$, that is, $C_m^-\subseteq P^{\dagger}$.
It thus follows from the definition of~$m_0$ (cf. Claim~\ref{Cl:HowLargem}) that $B^+\subseteq P^{\dagger}$, that is, $P\not\subseteq B^+$.
Since~$B^+\in\Op(\cH)$, it follows that $P\cap B^+\subseteq P_*$.
%
%
%

%

Now suppose that $P\subseteq A^+$.
Since $P\cap B^+\subseteq P_*$, the inequalities $P\subseteq P_*\cup A^+$ and $P\cap B^+\subseteq P_*$ both hold, thus also $f(P)\leq f(P_*)\vee f(A^+)$ and $f(P)\wedge f(B^+)\leq f(P_*)$.
Since~$\sd_L$ is an $L$-valued difference operation on the range of~$f$ (cf. Lemma~\ref{L:diffS}), it follows from Lemma~\ref{L:a1-b1leqa-b} that $f(P)\sd_Lf(P_*)\leq f(A^+)\sd_Lf(B^+)$, which implies~\eqref{Eq:fACm} right away.

It remains to handle the case where $P\not\subseteq A^+$.
Due to the obvious containment $C_m^+\subseteq A^+\cup B^-$, we get $P\subseteq P_*\cup A^+\cup B^-$, thus, since~$P$ is join-prime in~$\Op(\cH)$, we get $P\subseteq B^-$, thus $f(P)\sd_Lf(P_*)\leq f(B^-)$, and thus,
by using the equation $f(B^+)\wedge f(B^-)=0$ and the inequality $f(A^+)\leq f(B^+)\vee\pI{f(A^+)\sd_Lf(B^+)}$,
 \begin{align*}
 f(A^+)\wedge\pI{f(P)\sd_Lf(P_*)}&\leq f(A^+)\wedge f(B^-)\\
 &\leq\pI{f(B^+)\wedge f(B^-)}\vee
 \pII{\pI{f(A^+)\sd_Lf(B^+)}\wedge f(B^-)}\\
 &=\pI{f(A^+)\sd_Lf(B^+)}\wedge f(B^-)\\
 &\leq f(A^+)\sd_Lf(B^+)\,,
 \end{align*}
thus completing the proof of~\eqref{Eq:fACm} in the general case, and therefore of~\eqref{Eq:fA+Cm+A+B+}.

Now $A^+\subseteq B^+\cup(A^+\cap C_m^+)$, thus $A^+\sd_{\Ops(\cH_m)}B^+\subseteq A^+\cap C_m^+$, and thus
 \[
 g(A^+\sd_{\Ops(\cH_m)}B^+)\leq g(A^+\cap C_m^+)
 =f(A^+)\wedge g(C_m^+)\leq f(A^+)\sd_Lf(B^+)\,.
 \]
Since $f(A^+)\leq f(B^+)\vee g(A^+\sd_{\Ops(\cH_m)}B^+)$, the converse inequality
 \[
 f(A^+)\sd_Lf(B^+)\leq g(A^+\sd_{\Ops(\cH_m)}B^+)
 \]
holds, and therefore $f(A^+)\sd_Lf(B^+)=g(A^+\sd_{\Ops(\cH_m)}B^+)$.
\end{proof}

Lemma~\ref{L:1stepForcCl} deals with closure defects of the form $f(A^+)\leq f(B^+)\vee\gamma$.
A finite iteration of that result will yield our next lemma, which extends it to closure defects of the form $f(U)\leq f(V)\vee\gamma$, for arbitrary $U,V\in\Ops(\cH)$.

\begin{lemma}\label{L:2stepForcCl}
Let~$\gL$ be an additive subgroup of~$\EE'$.
Let~$\cH$ be a finite subset of $\cH_{\gL}\eqdef\setm{\ker(x)}{x\in\gL\setminus\set{0}}$, let~$L$ be a completely normal distributive lattice with zero, let $f\colon\Op(\cH)\to L$ be a $0$-lattice homomorphism, let $U,V\in\Ops(\cH)$, and let~$\gamma\in L$ such that $f(U)\leq f(V)\vee\gamma$.
Then there are a finite subset~$\widetilde{\cH}$ of~$\cH_{\gL}$, containing~$\cH$, $W\in\Ops(\widetilde{\cH})$, and a lattice homomorphism $g\colon\Op(\widetilde{\cH})\to L$ extending~$f$, such that $U\subseteq V\cup W$ and $g(W)\leq\gamma$.
\end{lemma}

\begin{proof}
We may assume that~$\cH$ is nonempty.
Fix an enumeration~$(A_0,B_0)$, \dots, $(A_{n-1},B_{n-1})$ of all pairs of open half-spaces with boundary in~$\cH$.
Since~$L$ is completely normal, there is a finite chain $S_0\subseteq S_1\subseteq\cdots\subseteq S_n$ of finite sublattices of~$L$ such that~$S_0$ contains $f[\Op(\cH)]\cup\set{\gamma}$ and~$S_i$ is consonant in~$S_{i+1}$ whenever $0\leq i<n$.
We construct inductively an ascending chain $\cH=\cH_0\subseteq\cH_1\subseteq\cdots\subseteq\cH_n$ of finite subsets of~$\cH_{\gL}$, together with an ascending chain of lattice homomorphisms $f_l\colon\Op(\cH_l)\to S_l$, for $0\leq l\leq n$, such that~$f_0=f$ and
 \begin{equation}\label{Eq:IndStepHkk+1}
 f_k\pI{A_l\sd_{\Ops(\cH_k)}B_l}\leq f(A_l)\sd_{S_1}f(B_l)
 \quad\text{whenever }0\leq l<k\leq n\,.
 \end{equation}
For $k=0$ there is nothing to verify.
Suppose having performed the construction up to level~$k$, with $0\leq k<n$.
By applying Lemma~\ref{L:1stepForcCl}, with~$\cH_k$ in place of~$\cH$, $f_k$ in place of~$f$, $S_{k+1}$ (which is a finite distributive lattice, thus, \emph{a fortiori}, a dual Heyting algebra) in place of~$L$, and $(A_k,B_k)$ in place of $(A^+,B^+)$, we get a finite subset~$\cH_{k+1}$ of~$\cH_{\gL}$, containing~$\cH_k$, together with a lattice homomorphism $f_{k+1}\colon\Op(\cH_{k+1})\to S_{k+1}$, extending~$f_k$, such that
 \[
 f_{k+1}\pI{A_k\sd_{\Ops(\cH_{k+1})}B_k}
 =f(A_k)\sd_{S_{k+1}}f(B_k)\,.
 \]
Since~$S_{k+1}$ contains~$S_1$, it follows that
 \begin{equation}\label{Eq:fk-stepk}
 f_{k+1}\pI{A_k\sd_{\Ops(\cH_{k+1})}B_k}
 \leq f(A_k)\sd_{S_1}f(B_k)\,.
 \end{equation}
Since~$\Ops(\cH_k)$ is a sublattice of~$\Ops(\cH_{k+1})$ and since~$f_{k+1}$ extends~$f_k$, it follows from the induction hypothesis~\eqref{Eq:IndStepHkk+1} (with fixed~$k$) that
 \[
 f_{k+1}\pI{A_l\sd_{\Ops(\cH_{k+1})}B_l}\leq
 f(A_l)\sd_{S_1}f(B_l)
 \quad\text{whenever }0\leq l<k\,,
 \]
and hence, by~\eqref{Eq:fk-stepk},
 \[
 f_{k+1}\pI{A_l\sd_{\Ops(\cH_{k+1})}B_l}\leq
 f(A_l)\sd_{S_1}f(B_l)
 \quad\text{whenever }0\leq l<k+1\,,
 \]
therefore completing the verification of the induction step.

At stage~$n$, we obtain a finite subset~$\widetilde{\cH}=\cH_n$ of~$\cH_{\gL}$, containing~$\cH$, together with a homomorphism $g=f_n\colon\Op(\cH_n)\to S_n$, extending~$f$, such that
 \begin{equation}\label{Eq:fn2Hn2}
 f_n(A_k\sd_{\Ops(\cH_n)}B_k)\leq
 f(A_k)\sd_{S_1}f(B_k)
 \quad\text{whenever }0\leq k<n\,.
 \end{equation}
%
\begin{figure}[htb]
 \[
 \def\labelstyle{\displaystyle}
 \xymatrix{
 \Ops(\cH)\ar@{_(->}[rrr]
 \ar[d]^{\text{consonant}}_{f\res_{\Ops(\cH)}
 =f_n\res_{\Ops(\cH)}}
 &&&
 \Ops(\cH_n)\ar[d]^{f_n\res_{\Ops(\cH_n)}}\\
 S_1\ar@{^(->}[rrr] &&& S_n
 }
 \]
\caption{Illustrating the proof of Lemma~\ref{L:2stepForcCl}}
\label{Fig:n2extensions}
\end{figure}
Since the open half-spaces with boundary in~$\cH$ generate~$\Ops(\cH)$ as a lattice (cf. Lemma~\ref{L:OpscH}) and since every pair of such half-spaces has the form $(A_k,B_k)$, it follows from Lemma~\ref{L:GenConsonant}, applied to~\eqref{Eq:fn2Hn2} and the commutative square represented in Figure~\ref{Fig:n2extensions}, that
 \[
 f_n(X\sd_{\Ops(\cH_n)}Y)\leq
 f(X)\sd_{S_1}f(Y)\,,
 \quad\text{for all }X,Y\in\Ops(\cH)\,.
 \]
In particular, $f_n(U\sd_{\Ops(\cH_n)}V)\leq f(U)\sd_{S_1}f(V)\leq\gamma$.
Let $W\eqdef U\sd_{\Ops(\cH_n)}\nobreak V$.
\end{proof}

%

\section{Enlarging the range of a homomorphism}
\label{S:ExtSpecial}

Until the end of this section we shall fix a set~$I$ and consider the vector space $\EE=\RR^{(I)}$ with basis~$I$, endowed with the coarsest topology making all canonical projections $\gd_i\colon\EE\to\RR$ (for $i\in I$) continuous.
We denote by~$\gL$ the additive subgroup of~$\EE'$ generated by $\setm{\gd_i}{i\in I}$ and we set (using the notation in Lemma~\ref{L:2stepForcCl}) $\cH_{\ZZ}\eqdef\cH_{\gL}$, the set of all \emph{integral hyperplanes} of~$\EE$.
We shall also set $\Delta_i\eqdef\ker\gd_i$, $\Delta_i^+\eqdef\bck{\gd_i>0}$, and $\Delta_i^-\eqdef\bck{\gd_i<0}$.
Any hyperplane $H\in\cH_{\ZZ}$ is the kernel of a nonzero element $x=\sum_{i\in I}x_i\gd_i\in\gL$, with all $x_i\in\ZZ$ and the support $\supp(x)\eqdef\setm{i\in I}{x_i\neq0}$ finite.
Since~$x$ is determined up to a nonzero scalar multiple, $\supp(x)$ depends of~$H$ only, so we may denote it by~$\supp(H)$.
For a set~$\cH$ of integral hyperplanes of~$\RR^{(I)}$, we shall set $\supp(\cH)\eqdef\bigcup\vecm{\supp(H)}{H\in\cH}$.

For $x\in\RR^{(I)}$ and $S\subseteq I$, we shall denote by $x\res_S$ the restriction of~$x$ to~$S$ extended by zero on $I\setminus S$.

\begin{lemma}\label{L:DetRes2S}
Let~$\cH$ be a set of integral hyperplanes of~$\RR^{(I)}$, with support~$S$, and let $Z\in\Bool(\cH)$.
Then $x\in Z$ if{f} $x\res_S\in Z$, for all $x\in\RR^{(I)}$.
\end{lemma}

\begin{proof}
For each $H\in\cH$, pick $p_H\in\gL$ with kernel~$H$, and set $H^+\eqdef\bck{p_H>0}$, $H^-\eqdef\bck{p_H<0}$.
Then for every $x\in\RR^{(I)}$, $x\in H^+$ if{f} $p_H(x)>0$, if{f} $p_H(x\res_S)>0$, if{f} $x\res_S\in H^+$.
The proof for~$H^-$ is similar.
Since the~$H^+$ and~$H^-$ generate~$\Bool(\cH)$ as a Boolean algebra, the general result follows easily.
\end{proof}

\begin{lemma}\label{L:AddIndepHyp}
Let~$\cH$ be a set of integral hyperplanes of~$\RR^{(I)}$ and let $i\in I\setminus\supp(\cH)$.
We denote by $\gf\colon\Op(\cH)\hookrightarrow\Op(\cH)\ast\sJ_2$ and $\psi\colon\Op(\cH)\hookrightarrow\Op(\cH\cup\set{\Delta_i})$ the diagonal embedding and the inclusion map, respectively, and we set
$\eps(X,Y,Z)\eqdef(X\cap\Delta_i^+)\cup(Y\cap\Delta_i^-)\cup Z$, for all $(X,Y,Z)\in\Op(\cH)\ast\sJ_2$.
Then~$\eps$ is an isomorphism and $\psi=\eps\circ\gf$.
\end{lemma}

We illustrate Lemma~\ref{L:AddIndepHyp} on Figure~\ref{Fig:AddIndepHyp}.

\begin{figure}[htb]
 \[
 \def\labelstyle{\displaystyle}
 \xymatrix{
 & \Op(\cH)\ar@{^(->}[dl]_{\gf}\ar@{_(->}[dr]^{\psi} & \\
 \Op(\cH)\ast\sJ_2\ar[rr]_{\eps} && \Op(\cH\cup\set{\Delta_i})
  }
 \]
\caption{Illustrating Lemma~\ref{L:AddIndepHyp}}
\label{Fig:AddIndepHyp}
\end{figure}

\begin{proof}
It is obvious that~$\gf$ and~$\psi$ are both $0,1$-lattice homomorphisms, that $\eps$ is lattice homomorphism (use Lemma~\ref{L:DotimesJ2}), and $\psi=\eps\circ\gf$.
Moreover, it follows from Corollary~\ref{C:ChangeH} that~$\eps$ is surjective.

Set $S\eqdef\supp(\cH)$.
In order to prove that~$\eps$ is one-to-one, it is sufficient to prove that every triple $(X,Y,Z)\in\Op(\cH)\ast\sJ_2$ is determined by the set $T\eqdef\eps(X,Y,Z)$.
Let $t\in\RR^{(I)}$.
Then $t\res_S\in\Delta_i$, thus $t\res_S\in T$ if{f} $t\res_S\in Z$, if{f} $t\in Z$ (cf. Lemma~\ref{L:DetRes2S});
hence~$T$ determines~$Z$.
Likewise, $t\res_S+\gd_i\in\Delta_i^+$, thus $t\res_S+\gd_i\in T$ if{f} $t\res_S+\gd_i$ belongs to $X\cup Z=X$, if{f} (using again Lemma~\ref{L:DetRes2S}) $t\res_S\in X$, if{f} $t\in X$.
Symmetrically, $t\res_S-\gd_i\in T$ if{f} $t\in Y$.
Therefore, $T$ determines both~$X$ and $Y$.
\end{proof}

\begin{lemma}\label{L:Adjoin1elt}
Let~$\cH$ be a set of integral hyperplanes of~$\RR^{(I)}$ and let $i\in I\setminus\supp(\cH)$.
Let~$L$ be a bounded distributive lattice, and let $a,b\in L$ such that $a\wedge b=0$.
Then every $0,1$-lattice homomorphism $f\colon\Op(\cH)\to L$ extends to a unique $0,1$-lattice homomorphism $g\colon\Op(\cH\cup\set{\Delta_i})\to L$ such that $a=g(\Delta_i^+)$ and $b=g(\Delta_i^-)$.
\end{lemma}

\begin{proof}
Keep the notation of Lemma~\ref{L:AddIndepHyp}.
A homomorphism $g\colon\Op(\cH\cup\set{\Delta_i})\to L$ satisfies the given conditions if{f} the homomorphism $h\eqdef g\circ\eps\colon\Op(\cH)\ast\sJ_2\to L$ satisfies $h(X,X,X)=f(X)$ whenever $X\in\Op(\cH)$, $a=h(\RR^{(I)},\es,\es)$, and $b=h(\es,\RR^{(I)},\es)$.
Apply Lemma~\ref{L:DotimesJ2}.
\end{proof}

\section{Representing countable completely normal lattices}
\label{S:ReprCtbleCN}

This section is devoted to a proof of our main theorem (Theorem~\ref{T:ReprCN}), together with a short discussion of some of its corollaries.

\begin{theorem}\label{T:ReprCN}
Every countable completely normal distributive lattice with zero is isomorphic to~$\Idc G$, for some Abelian $\ell$-group~$G$.
\end{theorem}

\begin{proof}
We must represent a countable completely normal distributive lattice~$L$ with zero.
The lattice~$\ol{L}$, obtained from~$L$ by adding a new top element, is also completely normal, and~$L$ is an ideal of~$\ol{L}$.
Any representation of~$\ol{L}$ as~$\Idc\ol{G}$, for an Abelian $\ell$-group~$\ol{G}$, yields $L\cong\Idc G$ for the $\ell$-ide\-al $G\eqdef\setm{x\in\ol{G}}{\seq{x}_{\ol{G}}\in L}$ (cf. Bigard, Keimel, and Wolfenstein \cite[\S~2.3]{BKW}).
Hence, it suffices to consider the case where~$L$ is bounded, following the strategy described in Section~\ref{S:Strategy}.
Fix a generating subset $\setm{a_n}{n\in\go}$ of~$L$.

As in Section~\ref{S:ExtSpecial}, we shall denote by~$\gL$ the additive subgroup of $\pI{\RR^{(\go)}}'$ generated by the canonical projections $\gd_n\colon\RR^{(\go)}\to\RR$ (where $n<\go$), and we shall denote by $\cH_{\ZZ}=\cH_{\gL}=\setm{H_n}{n\in\go}$ the set of all integral hyperplanes of~$\RR^{(\go)}$.
Moreover, let $\setm{(U_n,V_n,\gamma_n)}{n\in\go}$ be an enumeration of all triples $(U,V,\gamma)$, where $U,V\in\Ops(\cH_{\ZZ})$ and~$\gamma\in L$.

We construct an ascending chain $\vecm{\cH_n}{n\in\go}$ of nonempty finite subsets of~$\cH_{\ZZ}$, with union~$\cH_{\ZZ}$, together with an ascending sequence $\vecm{f_n}{n\in\go}$ of $0,1$-lattice homomorphisms $f_n\colon\Op(\cH_n)\to L$, as follows.

Take~$\cH_0\eqdef\set{\Delta_0}$ (cf. Section~\ref{S:ExtSpecial}); so $\Op(\cH_0)=\set{\es,\Delta_0^+,\Delta_0^-,\Delta_0^+\cup\Delta_0^-,\RR^{(\go)}}$ is isomorphic to~$\sJ_2$ (cf. Section~\ref{S:HomExt}).
Let $f_0\colon\Op(\cH_0)\to\set{0,a_0,1}$ be the unique homomorphism such that $f_0(\Delta_0^+)=a_0$, $f_0(\Delta_0^-)=0$, and $f_0\pI{\RR^{(\go)}}=1$.

Suppose $f_n\colon\Op(\cH_n)\to L$ already constructed.

Let $n=3m$ for some integer~$m$, denote by~$k$ the first nonnegative integer outside $\supp(\cH_n)$, and set $\cH_{n+1}\eqdef\cH_n\cup\set{\Delta_k}$.
By Lemma~\ref{L:Adjoin1elt}, there is a unique lattice homomorphism $f_{n+1}\colon\Op(\cH_{n+1})\to L$, extending~$f_n$, such that $f_{n+1}(\Delta_k^+)=a_m$ and $f_{n+1}(\Delta_k^-)=0$.
This will take care of the surjectivity of the restriction, to $\Ops(\cH_{\ZZ})$, of the union of the~$f_n$.

Let $n=3m+1$ for some integer~$m$, and set $\cH_{n+1}\eqdef\cH_n\cup\set{H_m}$.
Since~$L$ is completely normal and the range of~$f_n$ is finite, there is a finite sublattice~$S$ of~$L$ such that the range of~$f_n$ is consonant in~$S$.
By Lemma~\ref{L:OpcHCNpureinOpH'}, $f_n$ extends to a lattice homomorphism~$f_{n+1}$ from~$\Op(\cH_{n+1})$ to~$S$, thus to~$L$.
This will take care of the union of all~$f_n$ be defined on~$\Op(\cH_{\ZZ})$.

Let, finally, $n=3m+2$ for some integer~$m$.
By iterating Lemma~\ref{L:2stepForcCl} finitely many times, we get a finite subset~$\cH_{n+1}$ of~$\cH_{\ZZ}$ containing~$\cH_n$, together with an extension $f_{n+1}\colon\Op(\cH_{n+1})\to L$, such that for every $k\leq n$, if $\set{U_k,V_k}\subseteq\Ops(\cH_n)$ and $f_n(U_k)\leq f_n(V_k)\vee\gamma_k$, then $f_{n+1}(U_k\sd_{\Ops(\cH_{n+1})}V_k)\leq\gamma_k$.
This will take care of the union of the~$f_n$ be closed (cf. Definition~\ref{D:closed}) on $\Ops(\cH_{\ZZ})$.

The union~$f$ of all the~$f_n$ is a surjective lattice homomorphism from~$\Op(\cH_{\ZZ})$ onto~$L$.
Furthermore, the restriction~$f^-$ of~$f$ to~$\Ops(\cH_{\ZZ})$ is a closed, surjective lattice homomorphism from~$\Ops(\cH_{\ZZ})$ onto~$L$.
Now it follows from the Baker-Beynon duality (cf. Lemmas~3.2 and~3.3, and Section~7, in Baker~\cite{Baker1968}) that~$\Idc\FL(\go)$ is isomorphic to the sublattice of $\RR^{\RR^{(\go)}}$ generated by $\setm{\bck{f>0}}{f\in\gL}$, that is, $\Ops(\cH_{\ZZ})$ (cf. Lemma~\ref{L:OpscH}).
Hence, the map~$f^-$ induces a closed, surjective lattice homomorphism $g\colon\Idc\FL(\go)\twoheadrightarrow L$.
By Lemma~\ref{L:FactClosed}, this map factors through an isomorphism from $\Idc(\FL(\go)/I)$ onto~$L$, for a suitable $\ell$-ide\-al~$I$ of~$\FL(\go)$.
\end{proof}

Recall that Delzell and Madden's results in~\cite{DelMad1994} imply that Theorem~\ref{T:ReprCN} does not extend to the uncountable case.

\begin{corollary}\label{C:ReprCN1}
A second countable generalized spectral space~$X$ is homeomorphic to the $\ell$-spec\-trum of an Abelian $\ell$-group if{f} it is completely normal.
\end{corollary}

\begin{proof}
Since~$X$ is second countable, an easy application of compactness shows that~$\cKo(X)$ is countable.
Apply Theorem~\ref{T:ReprCN} and Lemma~\ref{L:Sp2Lat}.
\end{proof}

It is well known that the lattice~$\cC(G)$, of all convex $\ell$-subgroups of any $\ell$-group (not necessarily Abelian)~$G$, is the ideal lattice a completely normal distributive lattice with zero (see Iberkleid, Mart{\'{\i}}nez, and McGovern \cite[\S~1.2]{IMM2011} for a short overview).
Of course, in the Abelian case, $\cC(G)$ is isomorphic to the ideal lattice of~$\Idc{G}$.
A direct application of Theorem~\ref{T:ReprCN} yields the following.

\begin{corollary}\label{C:ReprCN2}
For every countable $\ell$-group~$G$, there exists a countable Abelian $\ell$-group~$A$ such that $\cC(G)\cong\cC(A)$.
\end{corollary}

The results of Kenoyer~\cite{Kenoy1984} and McCleary~\cite{McClear1986} imply that Corollary~\ref{C:ReprCN2} does not extend to the uncountable case.

The real spectrum~$\Specr{R}$, of any commutative unital ring~$R$, is a completely normal spectral space (cf. Coste and Roy~\cite{CosRoy1982}, Dickmann~\cite{Dickm1985}).
A direct application of Corollary~\ref{C:ReprCN1} yields the following.

\begin{corollary}\label{C:ReprCN3}
For every countable, commutative, unital ring~$R$, there exists a countable Abelian $\ell$-group~$A$ with unit such that $\Specr{R}\cong\Specl{A}$.
\end{corollary}

We prove in~\cite{MVRS} that Corollary~\ref{C:ReprCN3} does not extend to the uncountable case.

\section{Non-$\ell$-rep\-re\-sentabil\-ity results}
\label{S:NotEltary}

In this section we shall show that the class of $\ell$-rep\-re\-sentable distributive lattices is neither first-order, nor closed under infinite products (resp., homomorphic images).
All our non-$\ell$-rep\-re\-sentabil\-ity results will rely on the following concept.
We say that a distributive lattice~$D$ has \emph{countably based differences} if for all $a,b\in D$, the filter $a\ominus_Db$ (cf.~\eqref{Eq:Defominus}) is countably generated.
The following result is a restatement, in terms of lattices of principal $\ell$-ide\-als, of Cignoli, Gluschankof, and Lucas \cite[Thm.~2.2]{CGL}; see also Iberkleid, Mart{\'{\i}}nez, and McGovern \cite[Prop.~4.1.2]{IMM2011}.

\begin{lemma}\label{L:aleqb+cn}
Let~$G$ be an Abelian $\ell$-group.
Then the lattice~$\Idc G$ has countably based differences.\end{lemma}


\begin{examplepf}\label{Ex:IdcGnotdHagain}
A countable Abelian $\ell$-group~$G$, with unit, such that $\Idc{G}$ is not a dual Heyting algebra.
\end{examplepf}

\begin{proof}
Let~$G$ consist of all maps $x\colon\go\to\ZZ$ such that there are (necessarily unique) $\ga,\gb\in\ZZ$ such that $x(n)=\ga n+\gb$ for all large enough~$n$.
Then~$G$, ordered componentwise, is an $\ell$-subgroup of~$\ZZ^{\go}$.
The constant function~$a$, with value~$1$, and the identity function~$b$ on~$\go$, both belong to~$G^+$, $a+b$ is a unit of~$G$, and there is no least~$\bx\in\Idc G$ such that $\seq{b}\subseteq\seq{a}\vee\bx$.
\end{proof}

It is easy to see that the class of all $\ell$-rep\-re\-sentable distributive lattices is closed under finite cartesian products.
We shall now show that this observation does not extend to infinite products.

\begin{proposition}\label{P:nondHago}
Let~$D$ be a distributive lattice with zero.
If~$D$ is not a generalized dual Heyting algebra, then~$D^{\go}$ is not $\ell$-rep\-re\-sentable.
\end{proposition}

\begin{proof}
Denote by $\eps\colon D\hookrightarrow D^{\go}$ the diagonal embedding and suppose that~$D^{\go}$ is $\ell$-rep\-re\-sentable.
Since~$D$ is isomorphic to an ideal of~$D^{\go}$, it is also $\ell$-rep\-re\-sentable, thus, by Lemma~\ref{L:aleqb+cn}, $D$ has countably based differences.
On the other hand, since~$D$ is not a generalized dual Heyting algebra, there are $a,b\in D$ such that $a\ominus_Db$ has no least element.
The filter~$a\ominus_Db$ has a countably basis $\vecm{c_n}{n\in\go}$, with each $c_{n+1}\leq c_n$.

Now by Lemma~\ref{L:aleqb+cn}, $D^{\go}$ has countably based differences.
In particular, the filter $\eps(a)\ominus_{D^{\go}}\eps(b)$ has a countable basis $\vecm{\be_n}{n\in\go}$ with each $\be_{n+1}\leq\be_n$.
For all $n,k\in\go$, $a\leq b\vee\be_n(k)$, thus there exists $f(n,k)\in\go$ such that $c_{f(n,k)}\leq\be_n(k)$.
Set $\bx\eqdef\vecm{c_{f(n,n)+1}}{n\in\go}$.
Since $\eps(a)\leq\eps(b)\vee\bx$, there exists $n\in\go$ such that $\be_n\leq\bx$.
It follows that $c_{f(n,n)}\leq\be_n(n)\leq\bx(n)=c_{f(n,n)+1}$, a contradiction.
\end{proof}

By taking $D\eqdef\Idc{G}$, for the $\ell$-group of Example~\ref{Ex:IdcGnotdHagain}, we get

\begin{corollary}\label{C:nondHago}
The class of all $\ell$-rep\-re\-sentable bounded distributive lattices is not closed under infinite products.
\end{corollary}

Our next example involves the infinitary logic~$\scL_{\infty,\go}$, for which we refer the reader to Keisler and Knight~\cite{KeKn2004} (see also Bell~\cite{Bell2016}), of which we will adopt the terminology, in particular about back-and-forth families.
We say that a submodel~$M$, of a model~$N$, is an \emph{$\scL_{\infty,\go}$-el\-e\-men\-tary submodel} of~$N$, if for every~$\scL_{\infty,\go}$ sentence~$\gf$, with (finitely many, by definition of a sentence) parameters from~$M$, $M$ satisfies~$\gf$ if{f}~$N$ does.
Our example will show that there is no class of $\scL_{\infty,\go}$ sentences whose class of models is the one of all $\ell$-rep\-re\-sentable bounded distributive lattices.
As customary, we denote by~$\go_1$ the first uncountable ordinal.

\begin{examplepf}\label{Ex:Not1Ord}
A non-$\ell$-rep\-re\-sentable bounded distributive lattice~$\bD_{\go_1}$, of cardinality~$\aleph_1$, with a countable $\ell$-rep\-re\-sentable $\scL_{\infty,\go}$-el\-e\-men\-tary sublattice~$\bD_{\go,\go_1}$.
\end{examplepf}

\begin{proof}
For any sets~$I$ and~$J$ with $I\subseteq J$, we denote by~$\fin{I}$ the set of all finite subsets of~$I$, and we set
 \begin{align*}
 \bB_J&\eqdef\setm{X\subseteq J}
 {\text{either }X\text{ or }J\setminus X\text{ is finite}}\,,\\
 \bD_{I,J}&\eqdef\setm{(X,k)\in\bB_J\times\three}
 {(k=0\Rightarrow X\in\fin{I})\text{ and }
 (k\neq0\Rightarrow J\setminus X\in\fin{I})}\,,\\
 \bD_J&\eqdef\bD_{J,J}\,.
 \end{align*}
(Observe, in particular, that if~$J$ is finite, then $\bD_J=\bB_J\times\three$.)
We endow~$\bD_J$ and~$\bD_{I,J}$ with their componentwise orderings (i.e., $(X,k)\leq(Y,l)$ if $X\subseteq Y$ and $k\leq l$).
They are obviously bounded distributive lattices.
Further, we set
 \[
 \eps_{I,J}(X,k)\eqdef\begin{cases}
 (X,k)\,,&\text{if }k=0\,,\\
 (X\cup(J\setminus I),k)\,,&\text{if }k\neq0\,, 
 \end{cases}
 \quad\text{for any }(X,k)\in\bD_I\,.
 \]
For any sets~$I$ and~$J$ and any bijection $f\colon I\to J$, the map $\ol{f}\colon\bD_I\to\bD_J$, $(X,k)\mapsto(f[X],k)$ is a lattice isomorphism.
The following claim states some elementary properties of the maps~$\eps_{I,J}$ and~$\ol{f}$;
its proof is straightforward and we omit it.

\setcounter{claim}{0}

\goodbreak
\begin{claim}\label{Cl:Eltaryeps}\hfill
\begin{enumerater}
\item
For any sets $I\subseteq J$, $\bD_{I,J}$ is a bounded sublattice of~$\bD_J$ and~$\eps_{I,J}$ defines an isomorphism from~$\bD_I$ onto~$\bD_{I,J}$.

\item
The maps~$\eps_{I,J}$ form a direct system: that is, $\eps_{I,I}=\id_{\bD_I}$ and $\eps_{I,K}=\eps_{J,K}\circ\eps_{I,J}$ whenever $I\subseteq J\subseteq K$.

\item
For any set~$J$, the set~$\bD_J$ is the ascending union of all subsets~$\bD_{I,J}$, for $I\in\fin{J}$.

\item
Let $I'$, $I''$, $J'$, $J''$ be sets with $I'\subseteq I''$ and $J'\subseteq J''$, let $g\colon I''\to J''$ be a bijection with $g[I']=J'$, and let~$f$ be the domain-range restriction of~$g$ from~$I'$ onto~$J'$.
Then $\ol{g}\circ\eps_{I',I''}=\eps_{J',J''}\circ\ol{f}$.

\end{enumerater}
\end{claim}

For any set~$K$, we denote by~$\cL_K$ the first-order language obtained by adding to the language $(\vee,\wedge,0,1)$, of bounded lattices, a collection of constant symbols indexed by~$\bD_K$.
Then for every set~$I$ containing~$K$, the lattice~$\bD_I$ is naturally equipped with a structure of model for~$\scL_K$, by interpreting every $\ba\in\bD_K$ by $\eps_{K,I}(\ba)$.

For infinite sets~$I$ and~$J$, a finite subset~$K$ of~$I\cap J$, and finite sequences $(\bx_1,\dots,\bx_n)$ of elements of~$\bD_I$ and $(\by_1,\dots,\by_n)$ of elements of~$\bD_J$, let the statement $(\bx_1,\dots,\bx_n)\simeq_K(\by_1,\dots,\by_n)$ hold if there are $I'\in\fin{I}$ and $J'\in\fin{J}$ both containing~$K$, a bijection $f\colon I'\to J'$ extending the identity of~$K$, and elements $\bx'_1,\dots,\bx'_n\in\bD_{I'}$, such that each $\bx_i=\eps_{I',I}(\bx'_i)$ and each $\by_i=\eps_{J',J}(\ol{f}(\bx'_i))$.

\begin{claim}\label{Cl:simKBF}
The relation~$\simeq_K$ is a back-and-forth family for~$(\bD_I,\bD_J)$ with respect to the language~$\cL_K$.
\end{claim}

\begin{cproof}
Trivially, $\es\simeq_K\es$.
Further, if $(\bx_1,\dots,\bx_n)\simeq_K(\by_1,\dots,\by_n)$ holds \emph{via}~$I'$, $J'$, and~$f$ as above, then $\eps_{J',J}\circ\ol{f}\circ\eps_{I',I}^{-1}$ is an isomorphism from~$\bD_{I',I}$ onto~$\bD_{J',J}$, sending each~$\bx_i$ to~$\by_i$ and each $\eps_{K,I}(\bz)$, where~$\bz\in\bD_K$, to~$\eps_{K,J}(\bz)$; whence $(\bx_1,\dots,\bx_n)$ and $(\by_1,\dots,\by_n)$ satisfy the same quantifier-free formulas of~$\cL_K$.

Now let $(\bx_1,\dots,\bx_n)\simeq_K(\by_1,\dots,\by_n)$, \emph{via}~$I'$, $J'$, $f\colon I'\to J'$, and elements $\bx'_i\in\bD_{I'}$.
Let $\bx\in\bD_I$.
We need to find $\by\in\bD_J$ such that $(\bx_1,\dots,\bx_n,\bx)\simeq_K(\by_1,\dots,\by_n,\by)$.
There are a finite set~$I''$, with $I'\subseteq I''\subseteq I$, and $\bx''\in\bD_{I''}$, such that $\bx=\eps_{I'',I}(\bx'')$.
We set $\bx''_i=\eps_{I',I''}(\bx'_i)$ for each~$i$.
Since~$J$ is infinite, we can extend~$f$ to a bijection $g\colon I''\to J''$, with $J''\subseteq J$.
Then each $\bx_i=\eps_{I'',I}(\bx''_i)$ and (using Claim~\ref{Cl:Eltaryeps}) $\by_i=\eps_{J'',J}\pI{\ol{g}(\bx''_i)}$.
Hence, setting $\by\eqdef\eps_{J'',J}\pI{\ol{g}(\bx'')}$, we get the relation
 \begin{equation}\label{Eq:simeqext}
 (\bx_1,\dots,\bx_n,\bx)\simeq_K
 (\by_1,\dots,\by_n,\by)\,.
 \end{equation}
Symmetrically, for all $\by\in\bD_J$, there exists $\bx\in\bD_I$ such that~\eqref{Eq:simeqext} holds.
\end{cproof}

By Karp's Theorem (cf. Karp~\cite{Karp}, Barwise \cite[Thm.~VII.5.3]{Barw}, Keisler and Knight \cite[Thm.~1.2.1]{KeKn2004}), it follows that~$\bD_I$ and~$\bD_J$ satisfy the same $\scL_{\infty,\go}$-sentences of the language~$\cL_K$.
By letting~$I\eqdef\go$, $J\eqdef\go_1$ and by letting~$K$ range over all finite subsets of~$\go$, we thus obtain the following claim.

\begin{claim}\label{Cl:Lequivgogo1}
The lattice~$\bD_{\go,\go_1}$ $(\cong\bD_{\go})$ is an $\scL_{\infty,\go}$-el\-e\-men\-tary sublattice of~$\bD_{\go_1}$.
\end{claim}

Now we move to $\ell$-rep\-re\-sentabil\-ity.

\begin{claim}\label{Cl:DgoRepres}
Let~$I$ be countably infinite.
Then~$\bD_I$ is $\ell$-rep\-re\-sentable.
\end{claim}

\begin{cproof}
While Claim~\ref{Cl:DgoRepres} trivially follows from Theorem~\ref{T:ReprCN}, it is also easy to verify that~$\bD_I\cong\Idc{G}$ where~$G$ is the $\ell$-group of Example~\ref{Ex:IdcGnotdHagain}.
\end{cproof}

\begin{claim}\label{Cl:Dgo1notRepr}
The lattice~$\bD_{\go_1}$ does not have countably based differences.
In particular, it is not $\ell$-rep\-re\-sentable.
\end{claim}

\begin{cproof}
The elements $\ba\eqdef(\go_1,1)$ and $\bb\eqdef(\go_1,2)$ both belong to~$\bD_{\go_1}$.
Furthermore, the filter $\bb\ominus_{\bD_{\go_1}}\ba=\setm{(X,2)}{X\subseteq\go_1\text{ cofinite}}$ is not countably based.
The second part of our claim follows from Lemma~\ref{L:aleqb+cn}.
\end{cproof}

This claim finishes the proof of Example~\ref{Ex:Not1Ord}.
\end{proof}

\begin{note}
Denote by~$Z$ the completely normal spectral space constructed by Delzell and Madden in \cite[Thm.~2]{DelMad1994}.
Although there is an obvious $0,1$-lattice embedding from~$\bD_{\go_1}$ into~$\cKo(Z)$, it is not hard to see that the two lattices are not isomorphic.
Hence, $Z$ is not homeomorphic to the spectrum of~$\bD_{\go_1}$.
\end{note}

\begin{examplepf}\label{Ex:HomImNonl}
An $\ell$-rep\-re\-sentable bounded distributive lattice of cardinality~$\aleph_1$, with a non-$\ell$-rep\-re\-sentable lattice homomorphic image.
\end{examplepf}

\begin{proof}
The set~$\bD$, of all almost constant maps from~$\go_1$ to~$\three$, is a $0,1$-sublattice of~$\three^{\go_1}$.
It is straightforward to verify that $\bD\cong\Idc{H}$, where~$H$ denotes the Abelian $\ell$-group of all almost constant maps from~$\go_1$ to the lexicographical product of~$\ZZ$ by itself.
Now consider the non-$\ell$-rep\-re\-sentable lattice~$\bD_{\go_1}$ of Example~\ref{Ex:Not1Ord}.
The map $\rho\colon\bD\to\bD_{\go_1}$, $\bx\mapsto(\supp(\bx),x(\infty))$ is a surjective lattice homomorphism.
\end{proof}

By Stone duality, it follows that \emph{a spectral subspace of an $\ell$-spectrum may not be an $\ell$-spectrum}.

\section{Discussion}\label{S:Discussion}

\subsection{Ideal lattices of dimension groups}
\label{Su:IdDimGp}
A partially ordered Abelian group~$G$ is a  \emph{dimension group} if~$G$ is directed, unperforated (i.e., $mx\geq0$ implies that~$x\geq0$, whenever $x\in G$ and~$m$ is a positive integer), and~$G^+$ satisfies the Riesz refinement property (cf. Goodearl~\cite{Gpoag}).
The construction~$\Idc{G}$, for an Abelian $\ell$-group~$G$, extends naturally to arbitrary dimension groups, by replacing ``$\ell$-ide\-al'' by ``directed convex subgroup'' (in short \emph{ideal}).
However, now~$\Idc{G}$ is only a \jzs.
This semilattice is always distributive (i.e., it satisfies the Riesz refinement property), but it may not be a lattice.
In fact, \emph{every countable distributive \jzs\ is isomorphic to~$\Idc{G}$ for some countable dimension group~$G$} (this is stated in Goodearl and Wehrung \cite[Thm.~5.2]{GoWe2001}; it is also implicit in Bergman~\cite{Berg86});
\emph{moreover, the countable size is optimal} (Wehrung~\cite{CXAl1}).

In particular, it follows from Goodearl and Wehrung~\cite[Thm.~4.4]{GoWe2001} that for every distributive lattice~$L$ with zero, there exists a dimension group~$G$ such that $\Idc{G}\cong L$ (without any restriction on the cardinality of~$L$).
Attempting to infer, \emph{via} Theorem~1 of Elliott and Mundici~\cite{EllMun1993}, that if~$L$ is completely normal, then~$G$ is lattice-ordered, would already fail for the lattice $L\eqdef\bD_{\go_1}$ of Example~\ref{Ex:Not1Ord}, simply because~$\bD_{\go_1}$ is not $\ell$-rep\-re\-sentable.
The problem lies in the impossibility to read, on~$\Idc{G}$ alone, that every prime quotient of~$G$ be totally ordered, as illustrated by the following example (cf. \cite[p.~181]{EllMun1993}): let~$G$ be any non totally ordered simple dimension group (e.g., $G=\QQ\times\QQ$ with positive cone consisting of all $(x,y)$ with either $x=y=0$ or~$x>0$ and~$y>0$).
Then $\Idc{G}\cong\two$, yet~$G$ is not totally ordered.

\subsection{Lattices of $\ell$-ide\-als in non-Abelian $\ell$-groups}\label{Su:NonAbell}
It is proved in R\r{u}\v zi\v cka, T\r{u}ma, and Wehrung~\cite[Thm.~6.3]{RTW} that \emph{every countable distributive \jzs\ is isomorphic to~$\Idc{G}$ for some $\ell$-group~$G$; moreover, this result does not extend to semilattices of cardinality~$\aleph_2$}.
The gap at size~$\aleph_1$ is not filled yet.

\subsection{Open problems}\label{Su:OpPb}
Mellor and Tressl proved in~\cite{MelTre2012} that for any infinite cardinal~$\gl$, there is no~$\scL_{\infty,\gl}$ characterization of Stone duals of real spectra of commutative unital rings.
Our first open problem calls for an extension of that result to $\ell$-spectra, which would thus also extend the result of Example~\ref{Ex:Not1Ord} (where we get only~$\scL_{\infty,\go}$).

\begin{problem}\label{Pb:NotLinftyinfty}
Is the class of all $\ell$-rep\-re\-sentable lattices the class of all models of a class of~$\scL_{\infty,\gl}$ sentences, for some infinite cardinal~$\gl$?
\end{problem}

Recall from Example~\ref{Ex:HomImNonl} that a spectral subspace of an $\ell$-spectrum may not be an $\ell$-spectrum.
We also extend this result to real spectra in~\cite{MVRS}.
This suggests the following problem.

\begin{problem}\label{Pb:Retract}
Is every retract of an $\ell$-spectrum \pup{resp., real spectrum} also an $\ell$-spectrum \pup{resp., real spectrum}?
\end{problem}

The analogy between $\ell$-spec\-tra and real spectra (cf. Delzell and Madden~\cite{DelMad1995}), together with Corollary~\ref{C:ReprCN1}, suggests the following problem.

\begin{problem}\label{Pb:SepRealSpec}
Is every second countable completely normal spectral space homeomorphic to the real spectrum of some commutative, unital ring?
\end{problem}

The more general question, of characterizing real spectra of commutative, unital rings, is part of Problem~12 in Keimel's survey paper~\cite{Keim1995}.
Due to results by Mellor and Tressl~\cite{MelTre2012}, this is essentially hopeless without any cardinality restriction.

\section*{Acknowledgments}
The MV-space problem was introduced to me during a visit, to Universit\`a degli studi di Salerno, in September 2016.
Excellent conditions provided by the math department are greatly appreciated.

Uncountably many thanks are also due to Estella and Lynn, who stayed at my side during those very hard moments, and also to our family's cat Nietzsche, who had decided, that fateful November morning, that I shouldn't leave my desk.


\begin{thebibliography}{10}

\bibitem{Baker1968}
Kirby~A. Baker, \emph{Free vector lattices}, Canad. J. Math. \textbf{20}
  (1968), 58--66. \MR{0224524}

\bibitem{Barw}
Jon Barwise, \emph{{Admissible Sets and Structures}}, Springer-Verlag,
  Berlin-New York, 1975, An approach to definability theory, Perspectives in
  Mathematical Logic. \MR{0424560 (54 \#12519)}

\bibitem{Bell2016}
John~L. Bell, \emph{Infinitary logic}, The Stanford Encyclopedia of Philosophy
  (Edward~N. Zalta, ed.), Metaphysics Research Lab, Stanford University, winter
  2016 ed., 2016, accessible at the URL
  \url{https://plato.stanford.edu/entries/logic-infinitary/}.

\bibitem{Berg86}
George~M. Bergman, \emph{{Von Neumann regular rings with tailor-made ideal
  lattices}}, Unpublished note, available online at
  \url{http://math.berkeley.edu/\~{}gbergman/papers/unpub/}, October 26, 1986.

\bibitem{BKW}
Alain Bigard, Klaus Keimel, and Samuel Wolfenstein, \emph{{Groupes et {A}nneaux
  {R}\'eticul\'es}}, Lecture Notes in Mathematics, Vol. 608, Springer-Verlag,
  Berlin-New York, 1977. \MR{0552653 (58 \#27688)}

\bibitem{CGL}
Roberto Cignoli, Daniel Gluschankof, and Fran{\c{c}}ois Lucas, \emph{Prime
  spectra of lattice-ordered abelian groups}, J. Pure Appl. Algebra
  \textbf{136} (1999), no.~3, 217--229. \MR{1675803}

\bibitem{CigTor1996}
Roberto Cignoli and Antoni Torrens, \emph{The poset of prime {$\ell$}-ideals of
  an abelian {$\ell$}-group with a strong unit}, J. Algebra \textbf{184}
  (1996), no.~2, 604--612. \MR{1409232}

\bibitem{CosRoy1982}
Michel Coste and Marie-Fran\c{c}oise Roy, \emph{La topologie du spectre
  r\'eel}, Ordered fields and real algebraic geometry ({S}an {F}rancisco,
  {C}alif., 1981), Contemp. Math., vol.~8, Amer. Math. Soc., Providence, R.I.,
  1982, pp.~27--59. \MR{653174}

\bibitem{DavPri1990}
Brian~A. Davey and Hilary~A. Priestley, \emph{Introduction to {L}attices and
  {O}rder}, Cambridge Mathematical Textbooks, Cambridge University Press,
  Cambridge, 1990. \MR{1058437}

\bibitem{DelMad1994}
Charles~N. Delzell and James~J. Madden, \emph{A completely normal spectral
  space that is not a real spectrum}, J. Algebra \textbf{169} (1994), no.~1,
  71--77. \MR{1296582}

\bibitem{DelMad1995}
\bysame, \emph{Lattice-ordered rings and semialgebraic geometry. {I}}, Real
  analytic and algebraic geometry ({T}rento, 1992), de Gruyter, Berlin, 1995,
  pp.~103--129. \MR{1320313}

\bibitem{DinGri2004}
Antonio Di~Nola and Revaz Grigolia, \emph{Pro-finite {MV}-spaces}, Discrete
  Math. \textbf{283} (2004), no.~1-3, 61--69. \MR{2061482}

\bibitem{Dickm1985}
Maximo~A. Dickmann, \emph{Applications of model theory to real algebraic
  geometry. {A} survey}, Methods in mathematical logic ({C}aracas, 1983),
  Lecture Notes in Math., vol. 1130, Springer, Berlin, 1985, pp.~76--150.
  \MR{799038}

\bibitem{EllMun1993}
George~A. Elliott and Daniele Mundici, \emph{A characterisation of
  lattice-ordered abelian groups}, Math. Z. \textbf{213} (1993), no.~2,
  179--185. \MR{1221712}

\bibitem{Gpoag}
Kenneth~R. Goodearl, \emph{{Partially Ordered Abelian Groups with
  Interpolation}}, Mathematical Surveys and Monographs, vol.~20, American
  Mathematical Society, Providence, RI, 1986. \MR{845783 (88f:06013)}

\bibitem{GoWe2001}
Kenneth~R. Goodearl and Friedrich Wehrung, \emph{Representations of
  distributive semilattices in ideal lattices of various algebraic structures},
  Algebra Universalis \textbf{45} (2001), no.~1, 71--102. \MR{1809858}

\bibitem{GGFC}
George Gr{\"a}tzer, \emph{Lattice {T}heory. {F}irst {C}oncepts and
  {D}istributive {L}attices}, W. H. Freeman and Co., San Francisco, Calif.,
  1971. \MR{0321817}

\bibitem{LTF}
\bysame, \emph{{Lattice Theory: Foundation}}, Birkh\"auser/Springer Basel AG,
  Basel, 2011. \MR{2768581 (2012f:06001)}

\bibitem{IMM2011}
Wolf Iberkleid, Jorge Mart{\'{\i}}nez, and Warren~Wm. McGovern, \emph{Conrad
  frames}, Topology Appl. \textbf{158} (2011), no.~14, 1875--1887. \MR{2823701}

\bibitem{Johnst1982}
Peter~T. Johnstone, \emph{Stone {S}paces}, Cambridge Studies in Advanced
  Mathematics, vol.~3, Cambridge University Press, Cambridge, 1982. \MR{698074}

\bibitem{Karp}
Carol~R. Karp, \emph{Finite-quantifier equivalence}, Theory of {M}odels
  ({P}roc. 1963 {I}nternat. {S}ympos. {B}erkeley), North-Holland, Amsterdam,
  1965, pp.~407--412. \MR{0209132 (35 \#36)}

\bibitem{Keim1971}
Klaus Keimel, \emph{The {R}epresentation of {L}attice-{O}rdered {G}roups and
  {R}ings by {S}ections in {S}heaves},  (1971), 1--98. Lecture Notes in Math.,
  Vol. 248. \MR{0422107}

\bibitem{Keim1995}
\bysame, \emph{Some trends in lattice-ordered groups and rings}, Lattice theory
  and its applications ({D}armstadt, 1991), Res. Exp. Math., vol.~23,
  Heldermann, Lemgo, 1995, pp.~131--161. \MR{1366870}

\bibitem{KeKn2004}
H.~Jerome Keisler and Julia~F. Knight, \emph{Barwise: infinitary logic and
  admissible sets}, Bull. Symbolic Logic \textbf{10} (2004), no.~1, 4--36.
  \MR{2062240}

\bibitem{Kenoy1984}
David Kenoyer, \emph{Recognizability in the lattice of convex
  {$\ell$}-subgroups of a lattice-ordered group}, Czechoslovak Math. J.
  \textbf{34(109)} (1984), no.~3, 411--416. \MR{761423}

\bibitem{MaMu2002a}
Vincenzo Marra and Daniele Mundici, \emph{Combinatorial fans, lattice-ordered
  groups, and their neighbours: a short excursion}, S\'em. Lothar. Combin.
  \textbf{47} (2001/02), Article B47f, 19. \MR{1894026}

\bibitem{MaMu2002b}
\bysame, \emph{M{V}-algebras and abelian {$\ell$}-groups: a fruitful
  interaction}, Ordered algebraic structures, Dev. Math., vol.~7, Kluwer Acad.
  Publ., Dordrecht, 2002, pp.~57--88. \MR{2083034}

\bibitem{McClear1986}
Stephen~H. McCleary, \emph{Lattice-ordered groups whose lattices of convex
  {$\ell$}-subgroups guarantee noncommutativity}, Order \textbf{3} (1986),
  no.~3, 307--315. \MR{878927}

\bibitem{MelTre2012}
Timothy Mellor and Marcus Tressl, \emph{Non-axiomatizability of real spectra in
  {$\mathscr{L}_{\infty\lambda}$}}, Ann. Fac. Sci. Toulouse Math. (6)
  \textbf{21} (2012), no.~2, 343--358. \MR{2978098}

\bibitem{Mont1954}
Ant{\'o}nio~A. Monteiro, \emph{L'arithm\'etique des filtres et les espaces
  topologiques}, Segundo symposium sobre algunos problemas matem\'aticos que se
  est\'an estudiando en {L}atino {A}m\'erica, {J}ulio, 1954, Centro de
  Cooperaci\'on Cientifica de la UNESCO para Am\'erica Latina, Montevideo,
  Uruguay, 1954, pp.~129--162. \MR{0074805}

\bibitem{Mund2011}
Daniele Mundici, \emph{Advanced {{\L}}ukasiewicz {C}alculus and
  {MV}-{A}lgebras}, Trends in Logic---Studia Logica Library, vol.~35, Springer,
  Dordrecht, 2011. \MR{2815182}

\bibitem{RTW}
Pavel R\r{u}\v{z}i\v{c}ka, Ji\v{r}{\'\i} T\r{u}ma, and Friedrich Wehrung,
  \emph{Distributive congruence lattices of congruence-permutable algebras}, J.
  Algebra \textbf{311} (2007), no.~1, 96--116. \MR{2309879}

\bibitem{Schrij1986}
Alexander Schrijver, \emph{Theory of {L}inear and {I}nteger {P}rogramming},
  Wiley-Interscience Series in Discrete Mathematics, John Wiley \& Sons, Ltd.,
  Chichester, 1986, A Wiley-Interscience Publication. \MR{874114}

\bibitem{Sch1986}
Niels Schwartz, \emph{Real closed rings}, Algebra and order
  ({L}uminy-{M}arseille, 1984), Res. Exp. Math., vol.~14, Heldermann, Berlin,
  1986, pp.~175--194. \MR{891460}

\bibitem{Stone38a}
Marshall~H. Stone, \emph{{Topological representations of distributive lattices
  and Brouwerian logics}}, {\v{C}as. Mat. Fys.} \textbf{67} (1938), no.~1,
  1--25.

\bibitem{WDim}
Friedrich Wehrung, \emph{The dimension monoid of a lattice}, Algebra
  Universalis \textbf{40} (1998), no.~3, 247--411. \MR{1668068 (2000i:06014)}

\bibitem{CXAl1}
\bysame, \emph{Semilattices of finitely generated ideals of exchange rings with
  finite stable rank}, Trans. Amer. Math. Soc. \textbf{356} (2004), no.~5,
  1957--1970 (electronic). \MR{2031048 (2004j:06006)}

\bibitem{MVRS}
\bysame, \emph{{Real spectrum versus $\ell$-spectrum via Brumfiel spectrum}},
  hal-01550450, preprint, July 2017.

\end{thebibliography}

\providecommand{\noopsort}[1]{}\def\cprime{$'$}
  \def\polhk#1{\setbox0=\hbox{#1}{\ooalign{\hidewidth
  \lower1.5ex\hbox{`}\hidewidth\crcr\unhbox0}}}
\providecommand{\bysame}{\leavevmode\hbox to3em{\hrulefill}\thinspace}
\providecommand{\MR}{\relax\ifhmode\unskip\space\fi MR }
\providecommand{\MRhref}[2]{%
  \href{http://www.ams.org/mathscinet-getitem?mr=#1}{#2}
}
\providecommand{\href}[2]{#2}

\end{document}